\documentclass[12pt,leqno]{amsart}
\usepackage{latexsym,amsthm,amsmath,amssymb,mathrsfs,mathtools}
\usepackage{tikz}
\usetikzlibrary{arrows}
\usetikzlibrary{decorations.pathreplacing}
\usepackage{graphicx}
\usepackage{color}

\usepackage{setspace}

\title[Orientation preserving homeomorphisms with negative Jacobian]{Modulus of continuity of orientation preserving approximately differentiable
homeomorphisms with a.e. negative Jacobian}
\author{Pawe\l{} Goldstein}
\address{Pawe\l{} Goldstein\newline \indent Institute of Mathematics \newline \indent Faculty of Mathematics, Informatics and Mechanics\newline \indent University of Warsaw
\newline \indent Banacha 2\newline \indent 02-097 Warsaw, Poland\newline \indent {\tt goldie@mimuw.edu.pl}}
\author{Piotr Haj\l{}asz}
\address{Piotr Haj{\l}asz\newline \indent Department of Mathematics\newline \indent University of Pittsburgh\newline \indent 301
  Thackeray Hall\newline \indent Pittsburgh, PA 15260, USA\newline \indent {\tt hajlasz@pitt.edu}}
\thanks{P.G. was supported by FNP grant POMOST BIS/2012-6/3 \newline P.H.\ was supported by NSF grant DMS-1500647.}

\subjclass[2010]{Primary 46E35; Secondary 26B05, 26B10, 26B35, 74B20}
\keywords{approximately differentiable homeomorphisms, orientation preserving, H\"older condition, approximation}

plus 6pt minus 12pt
\def\eps{\varepsilon}

\def\id{{\rm id\, }}

\def\A{{\mathcal A}}

\def\C{{\mathcal C}}

\def\Q{{\mathcal Q}}

\newtheorem{theorem}{Theorem}
\newtheorem{lemma}[theorem]{Lemma}
\newtheorem{corollary}[theorem]{Corollary}
\newtheorem{proposition}[theorem]{Proposition}

\def\diam{{\rm diam\,}}

\theoremstyle{definition}
\newtheorem{remark}[theorem]{Remark}

\newcommand{\barint}{
\rule[.036in]{.12in}{.009in}\kern-.16in \displaystyle\int }

\newcommand{\barcal}{\mbox{$ \rule[.036in]{.11in}{.007in}\kern-.128in\int $}}

\newcommand{\bbbn}{\mathbb N}
\newcommand{\bbbr}{\mathbb R}

\def\ap{\operatorname{ap}}

\def\diam{\operatorname{diam}}

\def\dist{\operatorname{dist}}

\def\mvint_#1{\mathchoice
          {\mathop{\vrule width 6pt height 3 pt depth -2.5pt
                  \kern -8pt \intop}\nolimits_{\kern -3pt #1}}%
          {\mathop{\vrule width 5pt height 3 pt depth -2.6pt
                  \kern -6pt \intop}\nolimits_{#1}}%
          {\mathop{\vrule width 5pt height 3 pt depth -2.6pt
                  \kern -6pt \intop}\nolimits_{#1}}%
          {\mathop{\vrule width 5pt height 3 pt depth -2.6pt
                  \kern -6pt \intop}\nolimits_{#1}}}

\numberwithin{theorem}{section} \numberwithin{equation}{section}

\begin{document}

\begin{abstract}
  We construct an a.e. approximately differentiable homeomorphism of a unit $n$-dimensional cube onto itself which is
	orientation preserving, has the Lusin property (N) and has the Jacobian determinant negative a.e. Moreover, the
	homeomorphism together with its inverse satisfy a rather general sub-Lipschitz condition, in particular it can be
	bi-H\"older continuous with an arbitrary exponent less than $1$.
\end{abstract}

\maketitle

\section{Introduction}
\subsection{The main result}
It is well known that in the case of diffeomorphisms the sign of the Jacobian $J_\Phi$ carries topological information about the mapping $\Phi$
in the sense that it tells us whether the diffeomorphism is orientation preserving or orientation reversing. In fact, it is not difficult to prove,
using the notion of degree, that if $\Phi:\Omega_1\to\Omega_2$ is a homeomorphism between domains in $\bbbr^n$ and if $\Phi$ is differentiable at points
$x_1,x_2\in\Omega_1$, then the Jacobian of $\Phi$ cannot be positive at $x_1$ and negative at $x_2$, see \cite[Theorem 5.22]{HenclK}.
In particular, if a homeomorphism between Euclidean domains is differentiable a.e., then either $J_\Phi\geq 0$ a.e. or $J_\Phi\leq 0$ a.e.
On the other hand it is easy to construct a homeomorphism that is differentiable a.e. and has the Jacobian equal zero a.e., see \cite{Takacs} and references therein.

Applications to areas like nonlinear elasticity \cite{Ball,MullerS,Sverak}, the theory of quasiconformal and quasiregular mappings \cite{rickman}, or the theory of
mappings of finite distortion \cite{HenclK},
lead to study of homeomorphisms much
less regular than diffeomorphisms. Yet, it is still important to understand how the topological properties of these mappings are related to the sign of the Jacobian.
In particular, Sobolev homeomorphisms need not be differentiable a.e. in the classical sense, although they are weakly and approximately differentiable a.e.

A measurable function $f:E\to\bbbr$, defined on a measurable set $E\subset\bbbr^n$, is said to be {\em approximately differentiable} at $x\in E$ if there is a measurable set $E_x\subset E$
and a linear function $L:\bbbr^n\to\bbbr$ such that $x$ is a density point of $E_x$ and
$$
\lim_{E_x\ni y\to x}
\frac{|f(y)-f(x)-L(y-x)|}{|y-x|} = 0.
$$
The mapping $L$ is called the approximate derivative of $f$ at $x$ and it is denoted by $\ap Df(x)$. The approximate derivative is unique (if it exists).
If a mapping $\Phi:E\subset\bbbr^n\to\bbbr^n$ is approximately differentiable at $x\in E$, we define the approximate Jacobian as
$J_\Phi=\det\ap D\Phi(x)$.

We will be interested in mappings that are approximately differentiable a.e.

Diffeomorphisms map sets of measure zero to sets of measure zero, so it is natural to consider a similar property for classes of more general mappings.
We say that a mapping $\Phi:\Omega\to\bbbr^n$, defined on an open set $\Omega\subset\bbbr^n$, has the {\em Lusin property (N)} if it maps sets of Lebesgue
measure zero to sets of Lebesgue measure zero.

Thus homeomorphisms that are approximately differentiable a.e. and have the Lusin property are far reaching generalizations of diffeomorphisms and yet it turns out that for such
mappings the classical change of variables formula is true \cite{Federer44,FedererBook,HajlaszChange}.
Homeomorphisms that belong to the Sobolev space $W^{1,p}$ are approximately differentiable a.e. but they do not necessarily have the Lusin property.
The Lusin property is a strong additional condition that plays an important role in geometric applications of Sobolev mappings.

In this context Haj\l{}asz, back in 2001, asked the following questions:
(see \cite[Section~5.4]{HenclK} and \cite[p. 234]{HenclM}):

\noindent
{\sc Question~1.}
{\em Is it possible to construct a homeomorphism $\Phi:(0,1)^n\to\bbbr^n$ which is approximately differentiable a.e., has the Lusin property (N) and
at the same time $J_\Phi>0$ on a set of positive measure and $J_\Phi<0$ on a set of positive measure?}

\noindent
{\sc Question~2.}
{\em Is it possible to construct a homeomorphism $\Phi:[0,1]^n\to [0,1]^n$ which is approximately differentiable a.e., has the Lusin property (N),
equals to the identity on the boundary (and hence it is sense preserving in the topological sense), but $J_\Phi<0$ a.e.?}

\noindent
{\sc Question~3.}
{\em Is it possible to construct a homeomorphism $\Phi:(0,1)^n\to\bbbr^n$ of the Sobolev class $W^{1,p}$, $1\leq p<n-1$, such that at the same time
$J_\Phi>0$ on a set of positive measure and $J_\Phi<0$ on a set of positive measure?}

The answer to Question~1 is in the positive and it has been known to the authors since 2001,
but it has not been published until very recently. Namely, in the paper \cite{GH} the authors answered in the positive both questions 1 and 2.
The Question~3 has also been answered in a sequence of surprising and deep papers  \cite{CHT,HenclM,HenclV}.
For further motivation for the problems considered here we refer the reader to papers \cite{CHT,GH,HenclM,HenclV}, especially to
\cite{GH}, because the results proved here are strictly related to those in \cite{GH}.

The {\em uniform metric} in the space of homeomorphisms of the unit cube $Q=[0,1]^n$ onto itself is defined by
\begin{equation}
\label{um}
d(\Phi,\Psi)=\sup_{x\in Q}|\Phi(x)-\Psi(x)|+\sup_{x\in Q}|\Phi^{-1}(x)-\Psi^{-1}(x)|.
\end{equation}

The main result of \cite{GH} reads as follows.

\begin{theorem}
\label{main-previous}
There exists an almost everywhere approximately differentiable homeomorphism $\Phi$ of the cube $Q=[0,1]^n$ onto itself, such that
\begin{itemize}
 \item[(a)] $\Phi|_{\partial Q}=\id$,
 \item[(b)] $\Phi$ is measure preserving,
 \item[(c)] $\Phi$ is a limit, in the uniform metric $d$, of a sequence of measure preserving $\C^\infty$-diffeomorphisms of $Q$ that are identity on the boundary,
 \item[(d)] the approximate derivative of $\Phi$ satisfies
 \begin{equation}
 \label{E2}
\ap D\Phi(x) =
\left[
\begin{array}{ccccc}
1       &      0       &   \ldots   &   0      &     0      \\
0       &      1       &   \ldots   &   0      &     0      \\
\vdots  &   \vdots     &   \ddots   &  \vdots  &    \vdots  \\
0       &      0       &   \ldots   &   1      &     0      \\
0       &      0       &   \ldots   &   0      &    -1      \\
\end{array}
\right]
\quad
\text{a.e. in $Q$}.
 \end{equation}
 \end{itemize}
\end{theorem}

Note that (b) implies the Lusin condition (N).
The proof given in \cite{GH} does not give any estimates for the modulus of continuity of the homeomorphism $\Phi$. One of our main concerns in Theorem~\ref{main-previous} were the conditions (b) and (c).
Satisfying them required the use of results of Dacorogna and Moser \cite{DM}, which do not provide any reasonable estimates for the modulus of continuity. So what can we say about the regularity of $\Phi$
if we drop the conditions (b) and (c)?
While $\Phi$ cannot be
Lipschitz continuous (because Lipschitz continuous functions are differentiable a.e.), it is natural to ask how close the modulus of continuity of $\Phi$
can be to the Lipschitz one. For example, can $\Phi$ be H\"older continuous?

A positive answer is given in the next result which is the main result of the paper.

\begin{theorem}
\label{thm:main}
Assume $\phi:[0,\infty)\to[0,\infty)$ satisfies the following conditions:
\begin{itemize}
\item[(1)] $\phi$ is increasing, concave, continuous and $\phi(0)=0$,
    \medskip
\item[(2)]
$\displaystyle
\int_{0}^{1}\frac{ds}{\phi(s)}<\infty$.
\medskip
\item[(3)] $t \mapsto t^{-\alpha} \phi(t)$ is increasing on $(0,1)$ for some $\alpha\in(0,1)$.
\end{itemize}
Then there exists a homeomorphism $F$ of the cube $Q=[0,1]^{n}$ onto itself such that
\begin{itemize}
\item $F|_{\partial Q}=\id$,
\item $F$ has the Lusin property (N),
\item $F$ is approximately differentiable a.e.,
\item the approximate Jacobian $J_{F}$ is negative a.e. in $Q$,
\item for any $x,y\in Q$ we have $|F(x)-F(y)|+|F^{-1}(x)-F^{-1}(y)|\leq C(n,\phi)\phi(|x-y|)$.
\end{itemize}
\end{theorem}

Here $C(n,\phi)$ denotes a constant that depends on $n$ and $\phi$ only.

The function
$\phi(t)=Ct^\beta$, $\beta\in (0,1)$ satisfies properties (1), (2) and (3), so both $\Phi$ and $\Phi^{-1}$ can be $\beta$-H\"older continuous at the same time.
Also $\phi(t)=t\log^2 t$ satisfies (1) near zero and it can be extended to $[0,\infty)$  in a way that all conditions (1), (2) and (3) are satisfied.
It is easy to see that in that case $\Phi$ and $\Phi^{-1}$ are
H\"older continuous with any exponent $\beta<1$. In particular $\Phi$ and $\Phi^{-1}$ can be in the fractional Sobolev space $W^{s,p}$ for all $0<s<1$ and $1<p<\infty$.

Condition (2) is the main estimate describing the growth of the function $\phi$, while (1) simply means that $\phi$ is a modulus of continuity and (3) is of  technical nature.
Condition (3) implies that $\phi$ is  a modulus that is at least H\"older  with some positive exponent $\alpha$. Since, in applications, we are interested in moduli very
close to Lipschitz, the condition (3) is not restrictive.

The proof of Theorem~\ref{thm:main} involves lemmata of purely technical nature; they are collected in Section~\ref{aux}. In order to give motivation
for these technical arguments, we will describe now the
main ideas of the proof of Theorem~\ref{thm:main}.

\subsection{How to prove the main result}
The main building block in the proof of Theorem~\ref{thm:main} is the following result.
\begin{proposition}
\label{prop:main}
Assume $\phi:[0,\infty)\to[0,\infty)$ satisfies
\begin{itemize}
\item[(1)]  $\phi$ is increasing, concave, continuous and $\phi(0)=0$,
    \smallskip
\item[(2)]$\displaystyle
\int_{0}^{1}\frac{ds}{\phi(s)}<\infty$.

\end{itemize}
Then there exists an a.e. approximately differentiable homeomorphism $\Phi$ with the Lusin property (N) of the cube $Q=[0,1]^{n}$ onto itself, and a compact set $A$ in the interior of $Q$, such that
\begin{itemize}
\item $\Phi|_{\partial Q}=\id$,
\item $|A|>0$,
\item $\Phi$ is a reflection $(x_1,\ldots,x_{n-1},x_n)\mapsto (x_1,\ldots,x_{n-1},1-x_n)$ on $A$, $\Phi(A)=A$, and $\Phi$ is a $C^\infty$-diffeomorphism outside $A$,
\item at almost all points of the set $A$
\begin{equation}
\label{E4}
\ap D\Phi(x) =
\left[
\begin{array}{ccccc}
1       &      0       &   \ldots   &   0      &     0      \\
0       &      1       &   \ldots   &   0      &     0      \\
\vdots  &   \vdots     &   \ddots   &  \vdots  &    \vdots  \\
0       &      0       &   \ldots   &   1      &     0      \\
0       &      0       &   \ldots   &   0      &    -1      \\
\end{array}
\right]\, ,
\end{equation}
\item for any $x,y\in Q$ we have $|\Phi(x)-\Phi(y)|+|\Phi^{-1}(x)-\Phi^{-1}(y)|\leq C(n,\phi)\phi(|x-y|)$.
\end{itemize}
\end{proposition}

Note that we no longer require the technical condition (3) from Theorem \ref{thm:main}.

The set $A$ in Proposition~\ref{prop:main} is a Cantor set of positive measure defined in a standard way as the intersection of a nested family of cubes inside $Q$.
The homeomorphism $\Phi$ is a diffeomorphism outside $A$ and it is a mirror reflection when restricted to $A$.
It easily follows that $\Phi$ has the Lusin property and that \eqref{E4} is satisfied at all density points of $A$.

The homeomorphism $F$ from Theorem~\ref{thm:main} is constructed as a limit of a sequence of homeomorphisms $F_k$. The first homeomorphism $F_1=\Phi$
is defined as in Proposition~\ref{prop:main}. It has negative Jacobian in the set $K_1=A$.
Note that in the complement of $A$, $F_1$ is a diffeomorphism, so in a small neighborhood of each point in $Q\setminus A$, $F_1$ is almost affine.
Modifying $F_1$ slightly we can change it to a homoeomorphism $\tilde{F}_1$ that is affine
in many small cubes $\{Q_i\}_i$ in the complement of $A$.

A generic modification of a diffeomorphism in a way that it becomes affine in a neighborhood of a point is described in Lemma~\ref{goodapprox}.

Now, in each cube $Q_i$ we replace the affine map $\tilde{F}_1$ by a suitably rescaled version of the homeomorphism $\Phi$ from
Proposition~\ref{prop:main}; it is rescaled in a way
that it coincides with the affine map $\tilde{F}_1$ near the boundary of $Q_i$. The resulting mapping $F_2$ coincides with $F_1$ in a
neighborhood of $A$ so it has negative Jacobian on $A$. Also
in each cube $Q_i\subset Q\setminus A$, $F_2$ is a rescaled version of the homeomorphism from Proposition~\ref{prop:main} and hence it has 
negative Jacobian on a compact set $A_i\subset Q_i$ of positive measure.
Thus $F_2$ has negative Jacobian on a compact set $K_2=A\cup\bigcup_i A_i$. Clearly $K_1\subset K_2$.
The homeomorphism $F_3$ is constructed in a way that it coincides
with $F_2$ near $K_2=A\cup\bigcup_i A_i$ (and hence it has negative Jacobian on $K_2$) and in the complement of that set it has
negative Jacobian on compact subsets of many tiny cubes.
Thus $F_3$ has negative Jacobian on a compact set $K_3$ that contains $K_2$.
We construct homeomorphisms $F_4, F_5,\ldots$ and an increasing sequence of compact sets $K_4, K_5,\ldots$ in a similar manner.
The construction guarantees that $|Q\setminus \bigcup_{k=1}^\infty K_k|=0$.
The sequence of of homeomorphisms $\{ F_k\}$ is a Cauchy sequence with respect to the uniform metric $d$.
It is well known and easy to check that the space of homeomorphisms of $Q$ onto itself is complete with respect to the
uniform metric $d$ (see \cite[Lemma~1.2]{GH}) so the sequence $\{ F_k\}$ converges to a homeomorphism $F$.
Since $|Q\setminus \bigcup_{k=1}^\infty K_k|=0$ and $F_k$ has negative Jacobian on $K_k$,  it follows that $F$ has negative Jacobian almost everywhere.

Each of the homeomorphisms $F_k$ satisfies the continuity condition
\begin{equation}
\label{y40}
|F_k(x)-F_k(y)|+|F_k^{-1}(x)-F_k^{-1}(y)|\leq C_k\phi(|x-y|),
\end{equation}
because this is the estimate for $\Phi$ and $F_k$ contains many rescaled copies of $\Phi$. The problem is that
since the mappings became more and more complicated, the constant $C_k$ diverges to $\infty$ as $k\to\infty$ and hence the homeomorphism $F$
to which the homeomorphisms $F_k$ converge does {\em not} satisfy the desired estimate
$|F(x)-F(y)|+|F^{-1}(x)-F^{-1}(y)|\leq C\phi(|x-y|)$.

Note that in the argument above we never mentioned condition (3) from Theorem~\ref{thm:main}, since this condition is not needed for Proposition~\ref{prop:main}. Actually, condition (3) is used to overcome the problem with the
blow up of the estimates. Namely, we have
\begin{lemma}
Assume that
\begin{itemize}
\item[(1)] $\phi:[0,\infty)\to [0,\infty)$ is increasing, continuous, concave, $\phi(0)=0$,\\
\item[(2)] $\displaystyle \int_0^1 \frac{dt}{\phi(t)}<\infty$ and\\
\item[(3)] $t \mapsto t^{-\alpha} \phi(t)$ is increasing on $(0,1)$ for some $\alpha\in(0,1)$.
\end{itemize}
Then there exists $\psi$ such that
\begin{itemize}
\item[(a)] $\psi:[0,\infty)\to [0,\infty)$ is increasing, continuous, concave, $\psi(0)=0$,\\
\item[(b)]
$\displaystyle \int_0^1 \frac{dt}{\psi(t)}<\infty$ and\\
\item[(c)] $\displaystyle \lim_{t\to 0^+}\frac{\psi(t)}{\phi(t)}=0$.
\end{itemize}
\end{lemma}

The conditions (1), (2) and (3) are the same as in Theorem~\ref{thm:main} and conditions (a), (b) are the same as conditions (1), (2) in Proposition~\ref{prop:main}. However, (c) means that
$\psi(t)$ is much smaller than $\phi(t)$ when $t>0$ is sufficiently small.

Now, we construct a homeomorphism $\Phi$ as in Proposition~\ref{prop:main}, satisfying
$|\Phi(x)-\Phi(y)|+|\Phi^{-1}(x)-\Phi^{-1}(y)|\leq C\psi(|x-y|)$, and we use it in the definition of the sequence $\{ F_k\}$ constructed above. Property (c) of $\psi$
implies that
$|\Phi(x)-\Phi(y)|+|\Phi^{-1}(x)-\Phi^{-1}(y)|\leq \eps\phi(|x-y|)$, provided $x$ and $y$ are sufficiently close. That allows us to control the constant in the inequality \eqref{y40} so well that
we can take the constant $C_k$ on the right hand side of \eqref{y40} to be independent of $k$.
Passing to the limit shows that the limiting homeomorphism $F$ also satisfies \eqref{y40} with that constant.

\subsection{How to prove Proposition~\ref{prop:main}}
The set $A$ from the statement of Proposition~\ref{prop:main} is a Cantor set constructed in a standard way. Let $1=\alpha_0>\alpha_1>\alpha_2>\ldots$
be a decreasing sequence of positive numbers such that $2\alpha_{k+1}<\alpha_k$.
Let $\Q_0=Q$ be the initial cube of edge-length $\alpha_0=1$. Let $\Q_1$ be the union of $2^n$ cubes inside $Q$, each of edge-length $\alpha_1$. These $2^n$ cubes
are `evenly' distributed. Let $\Q_2$ be the union of $2^n\cdot2^n=2^{2n}$ cubes inside $\Q_1$, each of edge-length $\alpha_2$. Namely, inside each of the $2^n$ cubes in $\Q_1$ we have
$2^n$ smaller cubes of edge-length $\alpha_2$. Again, the cubes are `evenly' distributed, see Figure~\ref{fig:Phi2}. The condition $2\alpha_{k+1}<\alpha_k$ is necessary, as otherwise there would not be enough space for the cubes of the next generation. The Cantor set is then defined as $A=\bigcap_{k=0}^\infty \Q_k$. The volume of $\Q_k$ equals $2^{kn}\alpha_k^n=(2^k\alpha_k)^n$. Hence, in order for the Cantor set to have positive measure, we need
$\lim_{k\to\infty}2^k\alpha_k>0$.

The $2^n$ cubes $Q_j$, $j=1,2,\ldots,2^n$, of the first generation, that form the set $\Q_1$, are placed in two layers: $2^{n-1}$ cubes ($j=1,\ldots,2^{n-1}$) above the other $2^{n-1}$ cubes ($j=2^{n-1}+1,\ldots,2^n$). We choose indices in such a way that the cube $Q_{2^{n-1}+j}$ is right below the cube $Q_j$.

The homeomorphism $\Phi$ will be constructed as a limit of diffeomorphisms $\Phi_k$. The diffeomorphism $\Phi_1$ of $Q$ is identity near the boundary of $Q$
and it exchanges the cubes $Q_j$ from the top layer with the cubes $Q_{2^{n-1}+j}$ from the bottom layer. For each $j=1,2,\ldots,2^{n-1}$, the mapping $\Phi_1$ restricted to $Q_j$ is a translation of $Q_j$ onto
$Q_{2^{n-1}+j}$ and also $\Phi_1$ restricted to $Q_{2^{n-1}+j}$ is a translation of $Q_{2^{n-1}+j}$ onto
$Q_j$. This construction of $\Phi_1$ is carefully described in Lemma~\ref{est:lip}.

The diffeomorphism $\Phi_2$ coincides with $\Phi_1$ in $Q\setminus \Q_1$. Inside each cube $Q_j$ there is a family of $2^n$  cubes of the second generation. Now, the diffeomorphism $\Phi_2$
exchanges cubes from the top layer of this family with the cubes from the bottom layer. The diffeomorphism $\Phi_2$ does this in every cube $Q_j$, $j=1,2,\ldots,2^n$. The
diffeomorphisms $\Phi_k$ are defined in a similar way. The sequence $\{\Phi_k\}$ converges in the uniform metric to a homeomorphism $\Phi$.

The sequence of diffeomorphisms $\{\Phi_k\}$ reverses the vertical order of cubes used in the construction of the Cantor set $A$. Hence, the limiting homeomorphism $\Phi$ restricted to the Cantor set $A$
is the reflection in the hyperplane $x_n=1/2$. Thus $\Phi$ is approximately differentiable at the density points of $A$ and the approximate derivative satisfies \eqref{E4}. Also, it follows from the construction of the sequence
$\{\Phi_k\}$ that $\Phi$ is a diffeomorphism outside $A$, so it is differentiable there and the Lusin property of $\Phi$ easily follows.

We already pointed out that the numbers $\alpha_k$ used in this construction must satisfy $1=\alpha_0>\alpha_1>\ldots$,\, $2\alpha_{k+1}<\alpha_k$ and
$\lim_{k\to\infty}2^k\alpha_k>0$. That would be enough if we wanted to obtain a homeomorphism $\Phi$ with all properties listed in Proposition~\ref{prop:main} but the last one:
$|\Phi(x)-\Phi(y)|+|\Phi^{-1}(x)-\Phi^{-1}(y)|\leq C(n,\phi)\phi(|x-y|)$. This condition requires a much more careful choice of the sequence $\{\alpha_k\}$, since the numbers in the sequence must be related to
the function $\phi$. This is done in Lemma~\ref{ap:est ak}.

The numbers $\beta_k=(\alpha_{k-1}-2\alpha_k)/4$ are the distances of the cubes of the $k$-th generation to the boundaries of cubes of the ($k-1$)-th generation, see Figure~\ref{fig:Phi2}. Their properties are listed in
Lemma~\ref{ap:est bk}.

\subsection{Structure of the paper}
Section~\ref{aux} is of technical character: we recall the definition and basic properties of the modulus of continuity there and then we prove several lemmata needed in the proofs of Proposition~\ref{prop:main} and
Theorem~\ref{thm:main}. In Section~\ref{flipping} we prove Proposition~\ref{prop:main} and in Section~\ref{Pr1.2} we prove Theorem~\ref{thm:main}.

\section{Auxiliary lemmata}
\label{aux}

This section starts with a definition and basic properties of the modulus of continuity, but then  it is focused on
technical lemmata that will be used later in the proof of Proposition~\ref{prop:main} and Theorem~\ref{thm:main}. Some of these lemmata have already been mentioned in Introduction.
The reader may skip this section, go directly to Section~\ref{flipping} and come back to the results of Section~\ref{aux} whenever necessary.

We say that a continuous, non-decreasing and concave function $\phi:[0,\infty)\to [0,\infty)$ such that $\phi(0)=0$ is a modulus
of continuity of a function $f:X\to\bbbr$ defined on a metric space $(X,d)$ if
\begin{equation}
\label{e1}
|f(x)-f(y)|\leq\phi(d(x,y))
\quad
\text{for all $x,y\in X$.}
\end{equation}
It is well known that every continuous function on a compact metric space has a modulus of continuity.
The construction goes as follows. If $f$ is constant, then we take $\phi(t)\equiv 0$. Thus assume that $f$ is not constant.
Let
$\phi_1(t)=\sup\{|f(x)-f(y)|:\, d(x,y)\leq t\}$. Since the function $f$ is uniformly continuous and bounded,
the function $\phi_1:[0,\infty)\to [0,\infty)$ is non-decreasing and $\lim_{t\to 0^+}\phi_1(t)=0$.
Also, $\phi_1(t)$ is constant for $t\geq\diam X$, so \mbox{$M=\sup_{t>0}\phi_1(t)$} is positive and finite.
Clearly, $|f(x)-f(y)|\leq \phi_1(d(x,y))$ for all $x,y\in X$.
To obtain a modulus of continuity (as defined above), we define $\phi$ to be the least concave function greater than or equal to $\phi_1$.
Namely,
\begin{equation}
\label{e2}
\phi(t)=\inf\{\alpha t+\beta:\, \text{$\phi_1(s)\leq\alpha s+\beta$ for all $s\in [0,\infty)$}\}.
\end{equation}
It is easy to see that $\phi$ is concave and $\phi\geq\phi_1$, so \eqref{e1} is satisfied and $\phi$ is nonnegative.
Moreover, concavity implies continuity of $\phi$ on $(0,\infty)$, see
\cite[Theorem~A,~p.4]{RV}.
Since $\alpha\geq 0$ in \eqref{e2}, $\phi$ is non-decreasing. It remains to show that $\lim_{t\to 0^+}\phi(t)=0$.
For $0<\beta<M$ define $\alpha=\sup_{s>0}(\phi_1(s)-\beta)/s$. Note that $\alpha$ is positive and finite, because $\phi_1(s)-\beta<0$ for small $s$.
Clearly, $\phi_1(s)\leq\alpha s+\beta$ for all $s>0$.
Hence $\phi(t)\leq\alpha t+\beta$, so $0\leq\limsup_{t\to 0^+}\phi(t)\leq \beta$. Since $\beta>0$ can be arbitrarily small, we conclude that $\lim_{t\to 0^+}\phi(t)=0$.

The moduli $\phi(t)=Ct$ describe Lipschitz functions and more generally $\phi(t)=Ct^\alpha$, $\alpha\in (0,1]$, describe $\alpha$-H\"older continuous functions. In this paper we are interested in functions with modulus of continuity satisfying the following conditions:
\begin{itemize}
\item[(1)] $\phi:[0,\infty)\to [0,\infty)$ is increasing, continuous, concave, $\phi(0)=0$,\\
\item[(2)]
$\displaystyle \int_0^1 \frac{dt}{\phi(t)}<\infty$.
\end{itemize}
Here and in what follows, by an increasing function we mean a strictly increasing function. The Lipschitz modulus of continuity $\phi(t)=Ct$ does not satisfy the condition (2). However,
$\phi(t)=Ct^\alpha$, $\alpha\in (0,1)$, satisfies both conditions.
Also, there is a function $\phi$ that equals $\phi(t)=t\log^2 t$ near zero and has both properties (1) and (2).
A function with this modulus of continuity is
H\"older continuous with any exponent $\alpha<1$. Thus condition (2) means that $\phi$ is a sub-Lipschitz modulus of continuity and it can be pretty close to the Lipschitz one.

In order to avoid confusion we adopt the rule that $\phi^{-1}(t)$ will always stand for the inverse function and not for its reciprocal.

Although the following observation will not be used in the paper, one should note that (2) is
equivalent to $\phi^{-1}(t)/t$ satisfying the Dini condition. Indeed, according to Lemma~\ref{ap:estpsi}(b) we have $\lim_{t\to 0^+}t/\phi(t)=0$.
Since concave functions are locally Lipschitz, \cite[Theorem~A,~p.4]{RV}, integration by parts yields
\begin{equation*}
\begin{split}
\int_0^1\frac{dt}{\phi(t)}=\left.\frac{t}{\phi(t)}\right|_0^1+\int_0^1 \frac{t\phi'(t)dt}{\phi^2(t)}
=\frac{1}{\phi(1)}+\int_0^{\phi(1)}\frac{\phi^{-1}(s)/s\, ds}{s}\, .
\end{split}
\end{equation*}
The last equality follows from the substitution $s=\phi(t)$. Also, for the integration by parts to be rigorous,
we should integrate from $\eps$ to $1$ and then let $\eps\to 0^+$.

Now we will present some technical lemmata related to functions $\phi$ satisfying (1) and (2). They will be needed later in the proof of Proposition~\ref{prop:main}.
Since the remaining part of this section is of purely technical nature, the reader might want to skip it for now and return to it
when necessary.

In Lemmata~\ref{ap:estpsi}, \ref{ap:est ak}, \ref{ap:est bk} and~\ref{ap:est l} below we assume that $\phi:[0,\infty)\to [0,\infty)$
is a given function that satisfies conditions (1) and (2).

\begin{lemma}\leavevmode 
\label{ap:estpsi}
\begin{itemize}
\item[(a)] The function $\displaystyle t\mapsto \frac{t}{\phi(t)}$ is non-decreasing and
\item[(b)] $\displaystyle \lim_{t\to 0^+}\frac{t}{\phi(t)}=0.$
\end{itemize}
\end{lemma}
\begin{proof}
By concavity of $\phi$, if $s>t>0$,
\begin{equation}
\label{conc psi}
\phi(t)=\phi\Big(\frac{t}{s}s+\Big(1-\frac{t}{s}\Big)0\Big)\geq \frac{t}{s}\phi(s)+\Big(1-\frac{t}{s}\Big)\phi(0)=\frac{t}{s}\phi(s),
\end{equation}
which proves (a).
Next, (a) implies that the limit in (b) exists; assume, to the contrary, that $\lim_{t\to 0^+} t/\phi(t)=c>0$. Then, for all $t>0$, $1/\phi(t)\geq c/t$ and the improper integral in (2) is divergent. This proves (b).
\end{proof}

\begin{lemma}
\label{ap:est ak}
Let $N>0$ be such that
\begin{equation}
\label{eq2017}
\frac{1}{2^{N}}\left(1+\int_0^{\phi^{-1}(2^{-N})} \frac{ds}{\phi(s)}\right)=1.
\end{equation}
For $k=0,1,2,\ldots$ set
$$
\alpha_k=\frac{1}{2^{N+k}}\left(1+\int_0^{\phi^{-1}(2^{-N-k})} \frac{ds}{\phi(s)}\right)\, .
$$
Then 
\begin{itemize}
\item[(a)] $\alpha_0=1$,
\item[(b)] $2\alpha_{k+1}<\alpha_k$,
\item[(c)] $\lim_{k\to\infty}2^k \alpha_k=2^{-N}>0$,
\item[(d)] for all $k$ we have $\alpha_{k}\leq 2^{-k}$ and there exists $K_{o}$ such that
\begin{equation}\label{est_an}
\frac{1}{2^{N+k}}<\alpha_k<\frac{1}{2^{N+k-1}}
\quad
\text{for all $k\geq K_o$.}
\end{equation}
\end{itemize}
\end{lemma}
\begin{proof}
The existence of $N>0$ satisfying \eqref{eq2017} can be easily seen by taking limits of the
expression on the left hand side of \eqref{eq2017} as $N\to\infty$ and as $N\to 0^+$.
The properties (a)--(d) follow immediately from the definition of $\alpha_{k}$ and the fact that
$\int_{0}^{1}{ds}/{\phi(s)}<\infty$.
\end{proof}

\begin{lemma}
\label{ap:est bk}
For $k\geq 1$ define
$$
\beta_k=\frac{1}{4}\left(\alpha_{k-1}-2\alpha_k\right)=\frac{1}{2^{N+k+1}}\int_{\phi^{-1}(2^{-(N+k)})}^{\phi^{-1}(2^{-(N+k-1)})}\frac{ds}{\phi(s)}.
$$
Then
\begin{itemize}
\item[(a)] $\phi(2\beta_k)< 2^{-(N+k)+1}$,
\item[(b)] $2^{-(N+k)}\leq \phi(4\beta_k)$,
\item[(c)] the sequence $(\beta_{k})$ is decreasing and $\lim_{k\to\infty}\beta_{k}=0$.
\end{itemize}
\end{lemma}
\begin{proof}
Since, by assumptions, the function  $1/\phi(s)$ is decreasing, we have
\begin{equation}
\label{est:bk_up}
\begin{split}
\beta_k&< \frac{1}{2^{N+k+1}}\left(\phi^{-1}(2^{-(N+k-1)})-
\phi^{-1}(2^{-(N+k)})\right)2^{N+k}\\
&\leq \frac{1}{2}\phi^{-1}(2^{-(N+k-1)}),
\end{split}
\end{equation}
which proves (a).
Similarly,
\begin{equation}\label{est:bk_dn1}
\begin{split}
\beta_k&> \frac{1}{2^{N+k+1}}\left(\phi^{-1}(2^{-(N+k-1)})-\phi^{-1}(2^{-(N+k)})\right)2^{N+k-1}\\
&=\frac{1}{4}\left(\phi^{-1}(2^{-(N+k-1)})-2\phi^{-1}(2^{-(N+k)})\right)+\frac{1}{4}\phi^{-1}(2^{-(N+k)}).
\end{split}
\end{equation}
The function $\phi^{-1}$ is convex, $\phi^{-1}(0)=0$, thus for any $t>0$ we have
\begin{equation}\label{dblng}
\phi^{-1}(2t)\geq 2\phi^{-1}(t).
\end{equation}
Therefore, the expression $\frac{1}{4}\left(\phi^{-1}(2^{-(N+k-1)})-2\phi^{-1}(2^{-(N+k)})\right)$ is non-negative and ultimately
\begin{equation}\label{est:bk_dn}
\beta_k\geq \frac{1}{4}\phi^{-1}(2^{-(N+k)}),
\end{equation}
which proves (b).
Now, we claim that the sequence $(\beta_k)$ is decreasing. Indeed, by \eqref{est:bk_up},
\begin{equation*}
\begin{split}
\beta_k&<\frac{1}{2}\left(\phi^{-1}(2^{-(N+k-1)})-\phi^{-1}(2^{-(N+k)})\right)\\
&=\frac{1}{2^{N+k+1}} \frac{\phi^{-1}(2^{-(N+k-1)})-\phi^{-1}(2^{-(N+k)})}{2^{-(N+k-1)}-2^{-(N+k)}}.
\end{split}
\end{equation*}
Convexity of $\phi^{-1}$ implies that  the difference quotient
$$
\frac{\phi^{-1}(x)-\phi^{-1}(y)}{x-y}
$$
is a non-decreasing function of
both $x$ and $y$. Therefore
\begin{equation*}
\begin{split}
\beta_k &<\frac{1}{2^{N+k+1}} \frac{\phi^{-1}(2^{-(N+k-2)})-\phi^{-1}(2^{-(N+k-1)})}{2^{-(N+k-2)}-2^{-(N+k-1)}}\\
&=\frac{1}{4}(\phi^{-1}(2^{-(N+k-2)})-\phi^{-1}(2^{-(N+k-1)}))<\beta_{k-1},
\end{split}
\end{equation*}
where the last inequality follows from \eqref{est:bk_dn1}.
Finally, by taking $k\to \infty$ in (a) we see that $\lim_{k\to \infty} \beta_{k}=0$, because $\phi$ is increasing,
continuous and $\phi(0)=0$.
\end{proof}

\begin{lemma}
\label{ap:est l}
For $k\geq 1$ define $\lambda_{k}=\alpha_{k-1}/\beta_{k}$. Then there exists a constant $C=C(\phi)$ such that for all $k\geq 2$,
\begin{equation}
\label{eq:est l}
\sup_{\ell<k}\lambda_{\ell}\leq C(\phi)\lambda_{k},
\end{equation}
\end{lemma}
\begin{proof}
We start by proving \eqref{eq:est l} for all $k>K_{o}$,
where $K_{o}$ is given as in (d), Lemma \ref{ap:est ak}.

We have then
\begin{equation}
\lambda_{k}= \frac{\alpha_{k-1}}{\beta_{k}}\geq  \frac{2^{-N-k+1}}{\beta_{k}}
\geq  \frac{\phi(2\beta_{k})}{\beta_{k}}=2\frac{\phi(2\beta_{k})}{2\beta_{k}}
\end{equation}
and
\begin{equation}
\lambda_{k}\leq  \frac{2^{-N-k+2}}{\beta_{k}}\leq  4\frac{\phi(4\beta_{k})}{\beta_{k}}
= 16 \frac{\phi(4\beta_{k})}{4\beta_{k}}\leq   16\frac{\phi(2\beta_{k})}{2\beta_{k}},
\end{equation}
where the last inequality follows from concavity of $\phi$, see Lemma~\ref{ap:estpsi}(a).

Since both the sequence $(\beta_{\ell})$ and the function $\phi(t)/t$ are non-increasing,
the sequence $\phi(2\beta_\ell)/(2\beta_\ell)$ is non-decreasing so
\begin{equation}
\begin{split}
\sup_{\ell<k}\lambda_{\ell} &\leq \underbrace{\sup_{\ell\leq K_{o}} \lambda_{\ell}}_{=\Lambda}+\sup_{K_{o}<\ell<k } \lambda_{\ell}\\
&\leq \Lambda+16\sup_{K_{o}<\ell<k}\frac{\phi(2\beta_{\ell})}{2\beta_{\ell}}\leq \Lambda+16\frac{\phi(2\beta_{k})}{2\beta_{k}}\\
&\leq \Lambda+8\lambda_{k}\leq (\Lambda+8)\lambda_{k}=C(\phi)\lambda_{k},
\end{split}
\end{equation}
where the inequalities in the last line are justified by the fact that $\lambda_{k}>1$ for all $k$.
Obviously, the same (possibly with different $C(\phi)$) holds for $k\leq K_{o}$, since $\{\lambda_{1},\ldots,\lambda_{K_{o}}\}$ is a finite set of positive numbers.
\end{proof}

\begin{lemma}
\label{better}
Assume that
\begin{itemize}
\item[(a)] $\phi:[0,\infty)\to [0,\infty)$ is increasing, continuous, concave, $\phi(0)=0$,\\
\item[(b)] $\displaystyle \int_0^1 \frac{dt}{\phi(t)}<\infty$ and\\
\item[(c)] $t \mapsto t^{-\alpha} \phi(t)$ is increasing on $(0,1)$ for some $\alpha\in(0,1)$.
\end{itemize}
Then there exists $\psi$ such that
\begin{itemize}
\item[(1)] $\psi:[0,\infty)\to [0,\infty)$ is increasing, continuous, concave, $\psi(0)=0$,\\
\item[(2)]
$\displaystyle \int_0^1 \frac{dt}{\psi(t)}<\infty$ and\\
\item[(3)] $\displaystyle \lim_{t\to 0^+}\frac{\psi(t)}{\phi(t)}=0$.
\end{itemize}
\end{lemma}

\begin{proof}
Set $a_k=\phi^{-1}(2^{-k})$.

By an application of the Maclaurin-Cauchy integral test, we have (cf. Figure~\ref{fig1}, left)
$$
\sum_{k=1}^\infty 2^k (a_k-a_{k+1})\leq  \int_0^{a_1} \frac{dt}{\phi(t)}\leq \sum_{k=1}^\infty 2^{k+1}(a_k-a_{k+1}).
$$
Thus the convergence of the improper integral $\int_0^{a_1} \frac{dt}{\phi(t)}$ is equivalent to the convergence of the series $\sum _{k=1}^\infty 2^{k+1}(a_k-a_{k+1})$.

Denoting $A_k=2^{k+1} (a_k-a_{k+1})$ we have, by assumption (b), that $\sum_{k=1}^\infty A_k<\infty$.

One easily checks that $(A_k)^{-1}$ is equal to the slope of the secant of graph of the function $\phi$ through $(a_{k+1},2^{-(k+1)})$ and $(a_k,2^{-k})$ (cf. Figure~\ref{fig1}, right). Thus, by concavity of $\phi$, the sequence $(A_k)$ is non-increasing.

\begin{figure}
\begin{tikzpicture}[>=latex']

\filldraw[color=lightgray!30!white, fill=lightgray!30!white] (0.5,0) -- (0.5,1.41) -- (2,1.41) -- (2,0);
\filldraw[color=lightgray!60!white, fill=lightgray!60!white] (0.5,0) -- (0.5,0.707) -- (2,0.707) -- (2,0);
\draw[domain=0.08:4,samples=100,smooth,variable=\x] plot ({\x},{1/sqrt{\x}});
\draw[densely dotted] (0,1.41) -- (2,1.41);
\draw[densely dotted] (0,0.707) -- (2,0.707);
\draw[densely dotted] (0.5,1.41) -- (0.5,0);
\draw[densely dotted] (2,1.41) -- (2,0);
\node[left] at (0,0.707) {$2^k$};
\node[left] at (0,1.41) {$2^{k+1}$};
\node[below] at (0.5,0) {$a_{k+1}$};
\node[below] at (2,0) {$a_{k}$};
\node[above] at (3.5,0.6) {$\frac{1}{\phi(t)}$};

\draw[->] (-1,0) -- (4.2,0) node[right] {$t$};
\draw[->] (0,-1) -- (0,4.2) node[above] {$y$};
\begin{scope}[shift={(7,0)}]
\draw[domain=0.001:4,samples=100,smooth,variable=\x] plot ({\x},{pow(1024*\x,0.4)/8});

\draw[densely dotted] (0.53,0) -- (0.53,1.55);
\draw[densely dotted] (0,1.55) -- (0.53,1.55);
\draw[densely dotted] (3,0) -- (3,3.1);
\draw[densely dotted] (0,3.1) -- (3,3.1);
\draw[color=darkgray] (-0.5,0.904)-- (4,3.728);
\node[below] at (0.53,0) {$a_{k+1}$};
\node[below] at (3,0) {$a_{k}$};
\node[below] at (3.5,3.3) {$\phi(t)$};
\node[left] at (0,1.55) {$2^{-(k+1)}$};
\node[left] at (0,3.1) {$2^{-k}$};
\draw[->] (-1,0) -- (4.2,0) node[right] {$t$};
\draw[->] (0,-1) -- (0,4.2) node[above] {$y$};
\draw [decorate,decoration={brace,amplitude=4pt,mirror},xshift=2pt,yshift=0pt]
(0,1.55) -- (0,3.1) node [black,midway,xshift=0.5cm] {$\frac{1}{2^{k+1}}$};
\draw [decorate,decoration={brace,amplitude=4pt},xshift=0pt,yshift=2pt]
(0.53,0) -- (3,0) node [black,midway,yshift=0.3cm] {$a_k-a_{k+1}$};
\end{scope}

\end{tikzpicture}
\caption{} \label{fig1}
\end{figure}

Now, let
$$
A_k'=\sqrt{\sum_{\ell=k}^\infty A_\ell}-\sqrt{\sum_{\ell=k+1}^\infty A_\ell}=A_k\left(\displaystyle\sqrt{\sum_{\ell=k}^\infty A_\ell}+\sqrt{\sum_{\ell=k+1}^\infty A_\ell}\right)^{-1}
$$
Then
\begin{equation}
\label{2016}
\lim_{k\to\infty}\frac{A_k'}{A_k}=\lim_{k\to\infty}\left(\displaystyle\sqrt{\sum_{\ell=k}^\infty A_\ell}+\sqrt{\sum_{\ell=k+1}^\infty A_\ell}\right)^{-1}=\infty,
\end{equation}
by convergence of the series $\sum A_\ell$.

Also,
$$
\sum_{k=1}^\infty A_k'=\sqrt{\sum_{k=1}^\infty A_k}<\infty.
$$
However, the sequence $(A_k')$ need not be non-increasing.

To correct that, let us set $A_k''=\min_{\ell\leq k} A_\ell'$. The sequence $(A_k'')$ is, obviously, non-increasing,
moreover $A_k''\leq A_k'$, thus the series $\sum A_k''$ is convergent.

Next, we prove that $\lim_{k\to\infty} A_k''/A_k=\infty$.
Assume $\ell(k)\leq k$ is such that $A''_k=A'_{\ell(k)}$. Since $A_k'>0$ and $A'_k\to 0$, we have
$$
\forall_{m\in\bbbn}\ \exists_{k_o}\ \forall_{k>k_o}\qquad A_k'<A'_{\ell(m)}=\min\{A_1',\ldots,A_m'\}.
$$
Hence for $k>k_o$
$$
A_{\ell(k)}'=\min\{A_1',\ldots,A_k'\}\leq A_k'<A_{\ell(m)}'=\min\{A_1',\ldots,A_m'\}
$$
so $\ell(k)\not\in\{1,\ldots,m\}$ and thus $\ell(k)>m$. This shows that $\ell(k)\to\infty$ as $k\to\infty$.
Now \eqref{2016} yields
$$
\frac{A_k''}{A_k}=\frac{A_{\ell(k)}'}{A_k}\geq \frac{A_{\ell(k)}'}{A_{\ell(k)}}\xrightarrow{k\to\infty}\infty.
$$

Set $b_k=\displaystyle \sum_{\ell=k}^\infty \frac{A_\ell''}{2^{\ell+1}}$.

The sequence $(b_k)$ is decreasing and $\lim_{k\to\infty}b_k=0$. Moreover, we have $A_k''=2^{k+1}(b_k-b_{k+1})$.

Define $\psi$ as a piecewise-linear function, affine on all intervals $[b_{k+1},b_k]$ and such that
$\psi(b_k)=2^{-k}$.

For $t\in (b_{k+1},b_k)$, $\psi'(t)=1/A_k''$, thus $\psi'$ is non-increasing and positive, which implies that $\psi$ is concave and increasing.

The same argument through the Maclaurin-Cauchy integral test that, at the beginning of the proof, gave us the equivalence between the convergence of the series
$\sum_{k=1}^\infty A_k$ and the condition (b) for $\phi$, allows us to conclude the condition (2) from the convergence of the series $\sum_{k=1}^\infty A_k''$.
Also,
$$
\psi(0)=\lim_{t\to 0^+}\psi(t)=\lim_{k\to\infty}\psi(b_k)=\lim_{k\to\infty}2^{-k}=0.
$$

We still need to prove (3), i.e. that $\displaystyle \lim_{t\to 0^+} \frac{\psi(t)}{\phi(t)}=0$.

By the Stolz-Ces\`aro Theorem, \cite[Theorem~2.7.1]{CN},
$$
\lim_{k\to\infty}\frac{b_k}{a_k}=\lim_{k\to\infty}\frac{b_{k+1}-b_k}{a_{k+1}-a_k}=\lim_{k\to\infty}\frac{A_k''}{A_k}=\infty.
$$
Note that the condition (c) implies that whenever $k>1$ and $kt<1$, we have $\phi(kt)>k^\alpha \phi(t)$. Thus, for $t\in (b_{k+1},b_k]$,
where $k$ is sufficiently large
$$
0\leq \frac{\psi(t)}{\phi(t)}\leq
\frac{\psi(b_k)}{\phi(b_{k+1})}\leq
\frac{2^{-k}}{(b_{k+1}/a_{k+1})^\alpha \phi(a_{k+1})}=
2\left(\frac{a_{k+1}}{b_{k+1}}\right)^\alpha\xrightarrow{k\to\infty}0.
$$
Therefore $\displaystyle \lim_{t\to 0^+} \frac{\psi(t)}{\phi(t)}=0$.
\end{proof}

For a diffeomorphism $\Phi:\Omega\to\bbbr^{n}$ of class $C^{\infty}$, $\Omega\subset \bbbr^{n}$, we define
$$
\|D\Phi\|_{\Omega}=\sup_{x\in\Omega}\|D\Phi(x)\| =\sup_{x\in\Omega} \sup_{|\xi|=1}\left|D\Phi(x)\xi\right|,
$$
$$
\|(D\Phi)^{-1}\|_{\Omega}=\sup_{x\in\Omega}\|(D\Phi)^{-1}(x)\|=\sup_{x\in\Omega}\sup_{|\xi|=1}\left| (D\Phi(x))^{-1}\xi\right|.\\
$$

In order to obtain estimates on the modulus of continuity of the homeomorphism $\Phi$ constructed in Proposition~\ref{prop:main},
we need Lipschitz estimates for the building block of the construction: a `box-exchange' diffeomorphism that switches  top and bottom layers of
dyadic cubes of size $0<\alpha<\frac{1}{2}$ within the unit cube.

\begin{lemma}
\label{est:lip}
Assume that $Q_1,\ldots,Q_{2^n}$ are the closed $n$-dimensional cubes of edge-length $0<\alpha<\frac{1}{2}$
inside the unit cube $Q=[0,1]^n$,
with dyadic (i.e. with all coordinates equal $1/4$ or $3/4$) centers. Let $\beta=(1-2\alpha)/4$. Then there exists a smooth diffeomorphism $F_{\alpha}:Q\to Q$  such that
\begin{itemize}
\item[(a)] $F_{\alpha}$ exchanges the `top layer' cubes $Q_{j}$ with `bottom layer' cubes $Q_{2^{n-1}+j}$ in such a
way that the restriction of $F_{\alpha}$ to a $\beta/10$-tubular neighborhood of each of $Q_j$ is a translation;
\item[(b)] $F_{\alpha}$ is identity near $\partial Q$,
\item[(c)] $\|DF_{\alpha}\|_Q+\|D(F_{\alpha}^{-1})\|_Q<C(n)/\beta$ for some constant $C(n)$ dependent only on $n$.
\end{itemize}
\end{lemma}

\begin{proof}
Suppose that we already constructed $F_\alpha$ for $\alpha\in[\frac{1}{4},\frac{1}{2})$. If $0<\alpha<\frac{1}{4}$,
we set $F_\alpha=F_{\frac{1}{4}}$.

Since the dyadic cubes $Q_j$ of edge-length $\alpha$ are contained in the dyadic cubes of edge-length $1/4$, the mapping
$F_\alpha=F_{\frac{1}{4}}$ will exchange the cubes $Q_j$, so it will have the properties (a) and (b). The estimate (c)
follows from the corresponding estimate for $F_{\frac{1}{4}}$:
$$
\Vert DF_\alpha\Vert_Q +\Vert D(F_\alpha^{-1})\Vert_Q =
\Vert DF_{\frac{1}{4}}\Vert_Q +\Vert D(F_{\frac{1}{4}}^{-1})\Vert_Q <
8C(n)<\frac{2C(n)}{\beta},
$$
since $\beta<\frac{1}{4}$ when $0<\alpha<\frac{1}{4}$.

Thus it remains to show the construction of the mapping $F_\alpha$ when $\frac{1}{4}\leq\alpha<\frac{1}{2}$.

First, we sketch the construction of a smooth diffeomorphism $G_{\gamma}:[-1,1]^{n}\to [-1,1]^{n}$, $\gamma\in (0,\frac{1}{2}]$, such that
\begin{itemize}
\item[(A)] $G_{\gamma}$ is identity near $\partial [-1,1]^{n}$,
\item[(B)] $G_{\gamma}$ maps the cube $[-(1-\gamma),1-\gamma]^{n}$ to $[-\frac{1}{2},\frac{1}{2}]^{n}$
\item[(C)] $G_{\gamma}$ acts on  $[-(1-0.9\gamma),1-0.9\gamma]^{n}$ as a homogeneous affine scaling transformation (i.e. a homothety),
\item[(D)] $DG_{\gamma}$ is bounded by $C(n)/\gamma$ for some constant $C(n)$ dependent only on $n$,
\item[(E)] $DG_{\gamma}^{-1}$ is bounded by $C(n)$, independently of $\gamma$.
\end{itemize}
Such a diffeomorphism can be constructed as follows: let $D:S^{n-1}\to\bbbr$ denote the distance of a point $(r,\vartheta)\in\partial [-1,1]^{n}$,
given in radial coordinates, to the origin, as a function of the spherical coordinate $\vartheta$:
$$
S^{n-1}\ni \vartheta\mapsto (r,\vartheta)\in \partial [-1,1]^{n}\mapsto r\in\bbbr.
$$
 Fix $\eps\ll\gamma$ and let $R_{\eps}:S^{n-1}\to\bbbr$ approximate $D$
smoothly from below:\\ $0<D(\vartheta)-R_{\eps}(\vartheta)<\eps$ for $\vartheta\in S^{n-1}$.

Set $r_{\eps}=(1-0.9\gamma+2\eps)R_{\eps}$; one immediately checks that $r_{\eps}$ approximates $(1-0.9\gamma)D$ from above.
Note that $R_{\eps}-r_{\eps}\approx \gamma$. Last, let $\xi_{\gamma}:[0,1]\to \bbbr$ be a smooth, increasing function such that
$$
\xi(t)=\begin{cases} \frac{1}{2(1-\gamma)} &\text{ for }t\in [0,\eps],\\
1 &\text{ for } t\in[1-\eps,1].
\end{cases}
$$
We can find such $\xi$ with $\xi'$ bounded independently of $\gamma$; note that $\xi(t)\geq \frac{1}{2(1-\gamma)}\geq\frac{1}{2}$ for all $t$.
Then we can define $G_{\gamma}$ in radial coordinates as
$$
G_{\gamma}(r,\vartheta)=\begin{cases}\id & \text{ for }r>R_{\eps}(\vartheta)\\
(r\xi\left(\frac{r-r_{\eps}(\vartheta)}{R_{\eps}(\vartheta)-r_{\eps}(\vartheta)}\right),\vartheta) & \text{ for }r_{\eps}(\vartheta)\leq r\leq R_{\eps}(\vartheta)\\
(\frac{r}{2(1-\gamma)},\vartheta)&\text{ for }r<r_{\eps}.
\end{cases}
$$
Then $G_{\gamma}$ is a smooth diffeomorphism satisfying conditions (A) to (C). To obtain (D) and (E), note that $G_{\gamma}$ is a radial map, thus to find bounds on
$DG_{\gamma}$ and $DG_{\gamma}^{-1}$ it is enough to estimate $\partial_{r}|G_{\gamma}|$ from above and below.
$$
\partial_{r}|G_{\gamma}|=
\xi\left(\frac{r-r_{\eps}(\vartheta)}{R_{\eps}(\vartheta)-r_{\eps}(\vartheta)}\right)+
r \xi'\left(\frac{r-r_{\eps}(\vartheta)}{R_{\eps}(\vartheta)-r_{\eps}(\vartheta)}\right)\frac{1}{R_{\eps}(\vartheta)-r_{\eps}(\vartheta)},
$$
and thus
$$
\frac{1}{2}\leq \partial_{r}|G_{\gamma}|\leq \frac{C(n)}{\gamma}
$$
for some constant $C(n)$ depending only on $n$. This proves (D) and (E).

Dividing $ Q$ into $2^{n}$ dyadic cubes of edge $1/2$ and applying (rescaled by the factor $1/4$ and translated) $G_{4\beta}=G_{1-2\alpha}$
to each of them we obtain a diffeomorphism $H_{\alpha}$ of $ Q$ onto itself that shrinks the $2^{n}$ cubes $Q_{i}$ of edge $\alpha$ to concentric cubes of edge $1/4$
(extending by the same homothetic dilation to their small neighborhoods) and that is equal to $\id$ near $\partial  Q$.
Obviously, $DH_{\alpha}$ and $DH_{\alpha}^{-1}$ satisfy analogous estimates as those for $DG_{\alpha}$ and $DG_{\alpha}^{-1}$.

Let $F$ denote a diffeomorphism of $ Q$ that exchanges the `top' cubes of edge-length $1/4$ with their `bottom' counterparts, together with their
$\frac{1}{40}$-tubular neighborhoods and that is equal to identity near  $\partial  Q$. Then $F_{\alpha}=H_{\alpha}^{-1}\circ F\circ H_{\alpha}$ satisfies the conditions (a) and (b).
Moreover,
\begin{equation*}
\begin{split}
\|DF_{\alpha}\|_Q+&\|DF_{\alpha}^{-1}\|_Q\\
&\leq \|DH_{\alpha}^{-1}\|_Q\cdot \|DF\|_Q\cdot \|DH_{\alpha}\|_Q+\|DH_{\alpha}^{-1}\|_Q\cdot \|DF^{-1}\|_Q\cdot \|DH_{\alpha}\|_Q\\
&\leq C(n)/\beta.
\end{split}
\end{equation*}
\end{proof}

In the proof of Proposition~\ref{prop:main} we shall use Lemma~\ref{est:lip} to construct $F_\alpha$ for $\alpha=\alpha_k^{-1}\alpha_{k+1}$,
where $\alpha_k$ are as in Lemma~\ref{ap:est ak}.
Since
$$
\beta=\frac{1}{4}(1-2\alpha) =\frac{1}{4}\left(1-\frac{2\alpha_{k+1}}{\alpha_k}\right) = \frac{\beta_{k+1}}{\alpha_k}
$$
we have
\begin{corollary}
\label{cor:2.8}
Assume $\alpha_k$ are as in Lemma \ref{ap:est ak}.
Then for any $k=0,1,2,\ldots$ there exists a smooth diffeomorphism $F_{\alpha}$, with $\alpha=\alpha_k^{-1}\alpha_{k+1}$, that satisfies
conditions (a) and (b) of Lemma \ref{est:lip} and moreover
$$
\|DF_{\alpha_k^{-1}\alpha_{k+1}}\|_Q+
\|D(F_{\alpha_k^{-1}\alpha_{k+1}}^{-1})\|_Q<C(n)\frac{\alpha_{k}}{\beta_{k+1}}\, .
$$
\end{corollary}

The next lemma is similar to \cite[Lemma 3.8]{GH}, but some estimates are new.  Let us recall the notation used in \cite{GH}.
$$
\|D^{2}\Phi\|_{\Omega}=\sup_{x\in\Omega}\|D^{2}\Phi(x)\|=
\sup_{x\in\Omega}\sup_{|\xi|=|\eta|=1}\left|\sum_{i,j=1}^{n} \frac{\partial^{2}\Phi}{\partial x_{i}\partial x_{j}}(x)\xi_{i}\eta_{j}\right|.
$$
It follows from Taylor's theorem that if $\Phi\in C^\infty$ and $x\in B(x_o,r)\subset\Omega$, then
\begin{equation}
\label{cds34}
|\Phi(x)-\Phi(x_o)-D\Phi(x_o)(x-x_o)|\leq \Vert D^2\Phi\Vert_\Omega |x-x_o|^2.
\end{equation}

\begin{lemma}
\label{goodapprox}
Let $G:\Omega\to\bbbr^{n}$, $\Omega\subset\bbbr^{n}$, be a $C^\infty$-diffeomorphism such that
$$
M=\|DG\|_{\Omega}+\|(DG)^{-1}\|_{\Omega}+\|D^{2}G\|_{\Omega}<\infty.
$$
Let $E=\overline{B}(x_{o},2r) \Subset \Omega$, $B=B(x_o,r)$, $D=\overline{B}(x_{o},r/2)$
and let
$$
T(x)=G(x_o)+DG(x_o)(x-x_0)
$$
be the tangent map to $G$ at $x_o$. If
\begin{equation}
\label{r small}
r<\left(10 (M+1)^{2}2^{\ell}\right)^{-1}\quad\text{ for some }\ell\in \bbbn,
\end{equation}
then
\begin{itemize}
\item[(a)] $\diam G(B)<2^{-\ell}$,
\item[(b)] $T(D)\subset G(B)\subset T(E)$,
\item[(c)] there is a  $C^\infty$-diffeomorphism $\tilde{G}$ which coincides with $G$ on
$\Omega\setminus B(x_{o},4r/5)$ and coincides with $T$ on $B(x_{o},3r/5)$ such that
\item[(d)] the mapping $\tilde{G}$ is bi-Lipschitz on $\overline{B}$ with the bi-Lipschitz constant  $\Lambda=2M$, i.e.
$$
\Lambda^{-1}|x-y|\leq|\tilde{G}(x)-\tilde{G}(y)|\leq \Lambda |x-y|
\quad
\text{for all $x,y\in \overline{B}$.}
$$
\end{itemize}
\end{lemma}
\begin{proof}[Sketch of the proof]
Arguments that are similar to those that appear in the proof of Lemma~3.8 in \cite{GH} will be sketched only;
for details we refer the reader to \cite{GH}.

Note that (a) follows immediately from the condition \eqref{r small} and the bound $\|DG\|_\Omega < M$
since
$$
|G(x)-G(y)|\leq \|DG\|_\Omega |x-y|\leq M\cdot 2r<2^{-\ell}
\quad
\text{for all $x,y\in B$.}
$$
In what follows, we assume, for simplicity, that $x_o=0$, i.e. that the balls $B$, $D$ and $E$ are centered at the origin.

The mapping $\tilde{G}$ in (c), interpolating between $G$ and $T$ is given as
$$
\tilde{G}(x)=T(x)+\phi(|x|/r) (G(x)-T(x))=T(x)+L(x),
$$
where $\phi\in C^\infty (\bbbr,[0,1])$ is non-decreasing, $\phi(t)=0$  for $t\leq 3/5$ and $\phi(t)=1$ for $t\geq 4/5$; $\|\phi'\|_\infty\leq 9$.
Clearly $\tilde{G}$ coincides with $G$ on
$\Omega\setminus B(x_{o},4r/5)$ and coincides with $T$ on $B(x_{o},3r/5)$, but in order to complete the proof of (c) we need to prove that
$\tilde{G}$ is a diffeomorphism.

Elementary calculations show that for $x\in \overline{B}$ we have $\|DL(x)\|<10Mr$, which in turn allows us, for $x,y\in\overline{B}$,
to estimate $|\tilde{G}(x)-\tilde{G}(y)|$ from below:
\begin{equation}
\label{eq: openness}
|\tilde{G}(x)-\tilde{G}(y)|\geq \frac{1}{2}|DG(0)(x-y)|.
\end{equation}
The inequality \eqref{eq: openness} shows that $\tilde{G}$ is injective on $\overline{B}$, which suffices to prove that $G$ is a homeomorphism.
Taking in $\eqref{eq: openness}$ $y=x+\tau v$ (for some arbitrary $v\in\bbbr^n$ and sufficiently small $\tau$) gives
$$
\frac{1}{2} |DG(0)\tau v|\leq |\tilde{G}(x)-\tilde{G}(x+\tau v)|,
$$
thus
$$
|D\tilde{G}(x) v|=\lim_{\tau\to 0}\left|\frac{\tilde{G}(x+\tau v)-\tilde{G}(x)}{\tau}\right|\geq \frac{1}{2}|DG(0)v|,
$$
and non-degeneracy of $DG(0)$ implies non-degeneracy of $D\tilde{G}(x)$, thus $\tilde{G}$ is a diffeomorphism. This completes the proof of (c).

Since the diffeomorphisms $G$ and $\tilde{G}$ agree on the boundary of the ball $B$, $G(B)=\tilde{G}(B)$ so
$T(D)=\tilde{G}(D)\subset \tilde{G}(B)=G(B)$,
which is the first inclusion in (b).

To prove the other inclusion, $G(B)\subset T(E)$, we argue as follows.
If $x\in B$, then by \eqref{cds34}
$$
|G(x)-T(x)|\leq \Vert D^2G\Vert_\Omega |x|^2<Mr^2.
$$
On the other hand, the distance between the ellipsoids $T(\partial E)$ and $T(B)$ is larger than $Mr^2$
\begin{equation}
\label{costam}
\dist \big(T(\partial E),T(B)\big)>Mr^2
\end{equation}
so $G(B)\cap T(\partial E)=\emptyset$ and hence $G(B)\subset T(E)$. To prove \eqref{costam},
observe that the distance between the ellipsoids is minimized along the shortest semi-axis
(this fact follows easily from geometric arguments or from the Lagrange multiplier theorem), i.e.
if $x\in\partial B$ is such that $T(x)\in T(\partial B)$ and $T(2x)\in T(\partial E)$ are on a
shortest semi-axis of the ellipsoid $T(\partial E)$, then
$$
\dist \big(T(\partial E),T(B)\big)=|T(2x)-T(x)|=|DG(0)x|\geq \frac{|x|}{\Vert (DG(0))^{-1}\Vert}>\frac{r}{M}>Mr^2.
$$

Finally, using the estimate $\Vert DL(x)\Vert \leq 10Mr$ and the fact that ${\|DT(x)\|=\|DG(0)\|\leq M}$ we immediately get
$$
\|D\tilde{G}(x)\|\leq \|DT(x)\|+\|DL(x)\|\leq (1+10r)M\leq 2M.
$$
Hence the Lipschitz constant of $\tilde{G}$ on $\overline{B}$ is bounded by $2M$.
To get an estimate for the Lipschitz constant of $\tilde{G}^{-1}$, we return to \eqref{eq: openness}:
\begin{equation*}
\begin{split}
2M |\tilde{G}(x)-\tilde{G}(y)|&\geq 2 \|(DG)^{-1}\|_\Omega \cdot \frac{1}{2}|DG(0)(x-y)|\\
&\geq |DG(0)^{-1}DG(0)(x-y)|=|x-y|,
\end{split}
\end{equation*}
which completes the proof of (d) and hence that of the lemma.
\end{proof}

\section{Proof of Proposition~\ref{prop:main}}
\label{flipping}

The proof is a slight reworking of the construction in \cite[Lemma 2.1]{GH}; the main difference is that we no longer require $\Phi$ to be measure preserving,
and in return we obtain estimates on the modulus of continuity of $\Phi$ and $\Phi^{-1}$.

Let the sequences $(\alpha_{k})$ and $(\beta_{k})$ be defined as in Lemma \ref{ap:est ak} and Lemma \ref{ap:est bk}.

Let  $Q_1^{k},\ldots,Q_{2^n}^{k}$ be the closed $n$-dimensional cubes inside the unit cube $Q=[0,1]^n$,
of edge-length $\alpha_{k-1}^{-1}\alpha_k<1/2$, with dyadic (i.e. with all coordinates equal $1/4$ or $3/4$) centers $q_1,\ldots,q_{2^n}$, $q_j=(q_{j,1},\ldots,q_{j,n-1},q_{j,n})$, such that $q_{2^{n-1}+j}=(q_{j,1},\ldots,q_{j,n-1},1-q_{j,n})$. That means the first $2^{n-1}$ cubes are in the top layer and the
last $2^{n-1}$ are in the bottom layer, right below the corresponding cubes from the upper layer.

Our construction is iterative. The starting point is the diffeomorphism $\Phi_{1}=F_{\alpha_0^{-1}\alpha_{1}}$, constructed in
Corollary \ref{cor:2.8}. This diffeomorphism rigidly rearranges (translates) cubes $Q_j^1$ of the edge-length
$\alpha_0^{-1}\alpha_1=\alpha_1$ inside the unit cube $Q$. Moreover, with  each of the cubes $Q_{j}^{1}$, $\Phi_{1}$
translates also its neighborhood consisting of all the points with distance less than $\beta_{1}/10$.

 The diffeomorphism $\Phi_2$
coincides with $\Phi_1$ on $Q\setminus\bigcup_{j=1}^{2^n} Q_j^1$, but in the interior of each cube $Q_j^1$, rearranged by the diffeomorphism $\Phi_1$,
$\Phi_2$ is a rescaled and translated version of the diffeomorphism $F_{\alpha_{1}^{-1}\alpha_{2}}$.
It rearranges $2^{2n}$ cubes of the edge-length
$\alpha_1\cdot\alpha_1^{-1}\alpha_2=\alpha_2$.
Since the diffeomorphism $F_{\alpha_1^{-1}\alpha_2}$ is identity near the boundary of the cube $Q$, the rescaled versions of it, applied to the cubes $Q^1_j$, are identity near boundaries of
these cubes and hence the resulting mapping $\Phi_2$ is a smooth diffeomorphism.
The diffeomorphism $\Phi_3$ coincides with $\Phi_2$ outside the $2^{2n}$ cubes of the second generation rearranged by $\Phi_2$ and it is
a rescaled and translated version of the diffeomorphism $F_{\alpha_{2}^{-1}\alpha_{3}}$ inside each of the cubes rearranged by the diffeomorphism $\Phi_2$. It rearranges $2^{3n}$ cubes of the edge-length
$\alpha_2\cdot\alpha_2^{-1}\alpha_3=\alpha_3$ etc.

The diffeomorphism $\Phi_{k}$ rearranges $2^{kn}$ $k$-th generation cubes $Q^{k}_{j}$ (together with their small neighborhoods). Denoting the union of all these $k$-th generation cubes by $\mathcal{Q}_k$, we obtain a descending family of compact sets. Let
$$
A=\bigcap_{k=1}^\infty \mathcal{Q}_k.
$$
By (c), Lemma \ref{ap:est ak}, $|\mathcal{Q}_k|=2^{kn}\alpha_k^n$ has positive limit, thus  $A$ is a Cantor set of positive measure.

Thanks to the fact that at each step the subsequent modifications leading from $\Phi_{k}$ to $\Phi_{k+1}$ happen only in the $k$-th generation cubes, of diameter $\sqrt{n}\alpha_{k}$,  the sequence $\Phi_{k}$ is convergent in the uniform metric in the space of homeomorphisms. Therefore, the limit mapping $\Phi$ is a homeomorphism.

On the Cantor set $A$, $\Phi$ acts as a reflection $(x_1,\ldots,x_{n-1},x_n)\mapsto (x_1,\ldots,x_{n-1},1-x_n)$, and outside $A$, $\Phi$ is a diffeomorphism (for each $x\in Q\setminus A$ there exists $k\in \bbbn$ such that $\Phi$ restricted to a small neighborhood of $x$ coincides with $\Phi_{k}$). Thus $\Phi$ is a.e. approximately differentiable and its approximate derivative is equal to \eqref{E4} in the density points of $A$. Also, $\Phi$ has the Lusin property.

We need yet to prove the continuity estimates for $\Phi$.

Let us note the following observations on the construction of subsequent generations of cubes in our example:
\begin{itemize}
\item Each $k$-th generation cube $Q^{k}_{j}$ has a well defined `ancestor' cube $\A^{\ell}(Q_{j}^{k})$ in generation $\ell$,
for $0\leq\ell\leq k$, i.e. for $\ell=0,1,\ldots,k$ there exists a unique $\ell$-th generation cube $\A^{\ell}(Q_{j}^{k})$ such that $Q_{j}^{k}\subset \A^{\ell}(Q_{j}^{k})$.
\item Whenever $\ell<k$, $\dist(Q_{j}^{k}, Q\setminus \A^{\ell}(Q^{k}_{j}))\geq \beta_{\ell+1}+\beta_{\ell+2}+\cdots+\beta_{k}$.
\item Denote the union of the family of all $\ell$-th generation cubes by $\Q_{\ell}$, for $\ell=0,1,\ldots$
Let $k\in\bbbn$ and fix $\ell<k$. The mapping $\Phi_{k}$, restricted to $\breve{Q}^{\ell}_{j}:=Q^{\ell}_{j}\setminus \Q_{\ell+1}$,
coincides with $F_{\alpha^{-1}_\ell\alpha_{\ell+1}}$, translated and rescaled by a factor of $\alpha_{\ell}$.
Since the rescaling is both in the domain and in the range by the same factor, Corollary \ref{cor:2.8} yields
$$
\Vert D\Phi_k\Vert_{\breve{Q}^{\ell}_{j}} =
\Vert DF_{\alpha_\ell^{-1}\alpha_{\ell+1}}\Vert_Q \leq C(n)\frac{\alpha_\ell}{\beta_{\ell+1}} = C(n)\lambda_{\ell+1}.
$$
Moreover, by construction, $\Phi_{k}$ in the set
$$
N^{\ell}_{j}:=\{x\in Q~:~0<\dist(x,Q^{\ell}_{j})<\beta_{\ell}/10\}
$$
is an isometry -- a translation.
Similarly, $\Phi_{k}$ is a translation on the $k$-th generation cubes. Therefore on these sets $\|D\Phi_k\|=1$  (cf. Figure~\ref{fig:Phi2}).

\end{itemize}
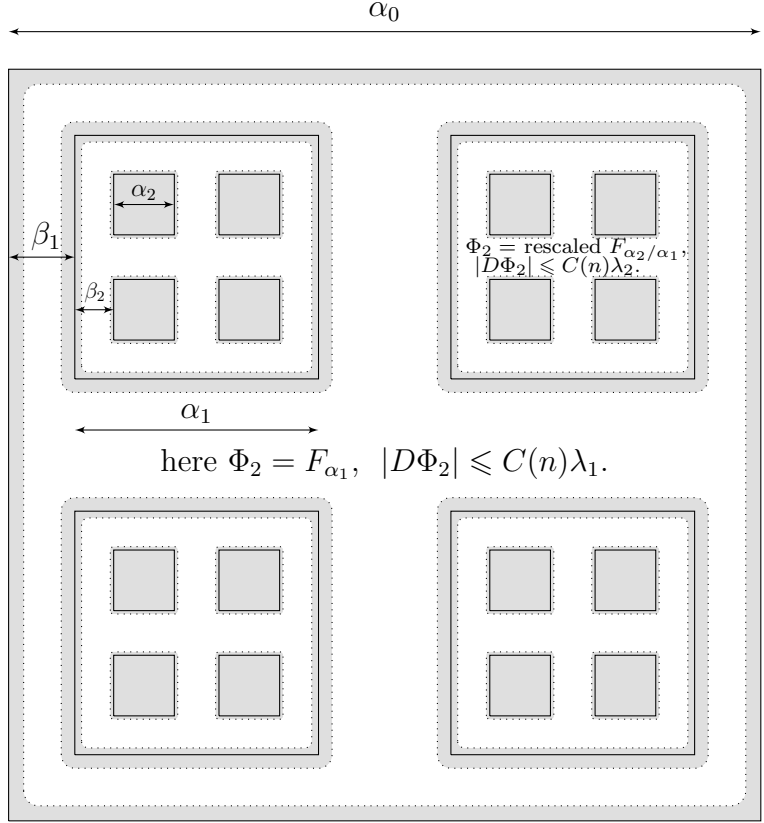
\begin{figure}
\caption{In white areas $\Phi_2$ is either $F_{\alpha_1}$ or a composition of a translation with a properly rescaled mapping $F_{\alpha_2/\alpha_1}$. In the shaded areas $\Phi_2$ is an isometry: translation or identity.}
\bigskip
\begin{tikzpicture}[>=latex']

\filldraw[fill=lightgray!50!white] (-5,-5)--(-5,5)--(5,5)--(5,-5)--(-5,-5);
\filldraw[dotted,fill=white] (-4.8,-4.6) --(-4.8,4.6) arc [radius=0.2, start angle=180, end angle=90] -- (4.6,4.8) arc [radius=0.2, start angle=90, end angle=0] -- (4.8,-4.6) arc [radius=0.2, start angle=0, end angle=-90] -- (-4.6,-4.8) arc [radius=0.2, start angle=-90, end angle=-180];

\foreach \x in {-2.5,2.5} {\foreach \y in {-2.5,2.5} {
\begin{scope}[shift ={(\x,\y)}, scale=0.9]
\filldraw[dotted,fill=lightgray!50!white] (-2,-1.8) --(-2,1.8) arc [radius=0.2, start angle=180, end angle=90] -- (1.8,2) arc [radius=0.2, start angle=90, end angle=0] -- (2,-1.8) arc [radius=0.2, start angle=0, end angle=-90] -- (-1.8,-2) arc [radius=0.2, start angle=-90, end angle=-180];
\filldraw[fill=lightgray!50!white] (-1.8,-1.8)--(1.8,-1.8)--(1.8,1.8)--(-1.8,1.8)--(-1.8,-1.8);
\filldraw[dotted,fill=white] (-1.7,-1.6) --(-1.7,1.6) arc [radius=0.1, start angle=180, end angle=90] -- (1.6,1.7) arc [radius=0.1, start angle=90, end angle=0] -- (1.7,-1.6) arc [radius=0.1, start angle=0, end angle=-90] -- (-1.6,-1.7) arc [radius=0.1, start angle=-90, end angle=-180];

\end{scope}
}}
\foreach \u in {-2.5,2.5} {\foreach \v in {-2.5,2.5} {
\begin{scope}[shift={(\u,\v)}, scale=0.28]
\foreach \x in {-2.5,2.5} {\foreach \y in {-2.5,2.5} {
\begin{scope}[shift ={(\x,\y)}, scale=0.8]
\filldraw[dotted,fill=lightgray!50!white] (-2,-1.8) --(-2,1.8) arc [radius=0.2, start angle=180, end angle=90] -- (1.8,2) arc [radius=0.2, start angle=90, end angle=0] -- (2,-1.8) arc [radius=0.2, start angle=0, end angle=-90] -- (-1.8,-2) arc [radius=0.2, start angle=-90, end angle=-180];
\filldraw[fill=lightgray!50!white] (-1.8,-1.8)--(1.8,-1.8)--(1.8,1.8)--(-1.8,1.8)--(-1.8,-1.8);
\end{scope}
}}
\end{scope}
}}
\draw[<->] (-5,5.5)--(5,5.5);
\node[above] at (0,5.5) {$\alpha_0$};
\draw[<->] (-4.12,0.2)--(-0.88,0.2);
\node[above] at (-2.5,0.15) {$\alpha_1$};
\draw[<->] (-3.6,3.2)--(-2.8,3.2);
\node[above, scale=0.8] at (-3.2,3.15) {$\alpha_2$};
\draw[<->] (-5,2.5)--(-4.12,2.5);
\node[above] at (-4.5,2.45) {$\beta_1$};
\draw[<->] (-4.12,1.8)--(-3.6,1.8);
\node[above,scale=0.7] at (-3.85,1.8) {$\beta_2$};

\node[above] at (0,-0.6) {here $\Phi_2=F_{\alpha_1}$,\,\, $|D\Phi_2|\leqslant C(n)\lambda_1$.};
\node at (2.5,2.6) {\begin{tiny} $\Phi_2=$ rescaled $F_{\alpha_2/\alpha_1}$,\end{tiny}};
\node at (2.3,2.37) {\begin{tiny}$|D\Phi_2|\leqslant C(n)\lambda_2$.\end{tiny}};

\end{tikzpicture}
\label{fig:Phi2}
\end{figure}

Next, we prove that, for any $\ell$, the diffeomorphism $\Phi_{\ell}$ satisfies the modulus of continuity estimate: for any $x,y\in Q$,
\begin{equation}\label{moc est}
|\Phi_{\ell}(x)-\Phi_{\ell}(y)|\leq C(n,\phi) \phi(|x-y|),
\end{equation}
with the constant $C=C(n,\phi)$ independent on $\ell$.

We prove it by induction. Setting $\Phi_{o}=\id$ we may assume that \eqref{moc est} holds for $\ell=0$. Indeed, according to Lemma~\ref{ap:estpsi}(a) we have
$$
|\Phi_o(x)-\Phi_o(y)|=|x-y|\leq \frac{\sqrt{n}}{\phi(\sqrt{n})}\phi(|x-y|).
$$
In the inductive step, assume that for some $k\in \bbbn$ the estimate \eqref{moc est} holds for all $\ell<k$.

Fix $x,y\in Q$, $x\neq y$. If both $x$ and $y$ lie outside the $(k-1)$-th generation cubes $\Q_{k-1}$, then, by the inductive assumption,
$$
|\Phi_{k}(x)-\Phi_{k}(y)|=|\Phi_{k-1}(x)-\Phi_{k-1}(y)|\leq C(n,\phi) \phi(|x-y|),
$$
since $\Phi_{k}$ and $\Phi_{k-1}$ coincide outside $\Q_{k-1}$. Thus, in what follows, we assume that $x\in Q^{k-1}_{j}$, where $Q^{k-1}_{j}\subset \Q_{k-1}$ is a $(k-1)$-th generation cube.

The sequence $(\beta_{i})$ is decreasing to $0$, thus either $|x-y|> \beta_{1}/10$, or there exists $m\in\bbbn$ such that $\beta_{m+1}/10<|x-y|\leq \beta_{m}/10$.

If $|x-y|> \beta_{1}/10$, then
$$
|\Phi_{k}(x)-\Phi_{k}(y)|\leq \diam Q=\sqrt{n}\leq \frac{\sqrt{n}}{\phi(\beta_{1}/10)}\phi(|x-y|)=C(n,\phi)\phi(|x-y|).
$$

If $\beta_{m+1}/10<|x-y|\leq \beta_{m}/10$  and $1<m\leq k-1$, then
$$
|x-y|\leq\frac{\beta_{m}}{10}<\beta_{m}+\cdots+\beta_{k-1}\leq \dist(Q_{j}^{k-1}, Q\setminus \A^{m-1}(Q_{j}^{k-1})).
$$
This shows that $x,y\in \A^{m-1}(Q_{j}^{k-1})$ and
\begin{equation}
\label{est by am-1}
|\Phi_{k}(x)-\Phi_{k}(y)|\leq \diam (\Phi_{k}(\A^{m-1}(Q_{j}^{k-1})))=\sqrt{n}\alpha_{m-1},
\end{equation}
because $\Phi_{k}(\A^{m-1}(Q_{j}^{k-1}))$ is again one of the $(m-1)$-th generation cubes.

Last, if $\beta_{m+1}/10<|x-y|\leq \beta_{m}/10$  and $m\geq k$, then
$$
\dist(y,Q_{j}^{k-1})\leq |x-y|\leq \frac{\beta_m}{10}\leq \frac{\beta_{k}}{10}< \frac{\beta_{k-1}}{10},
$$
thus either $y\in Q_{j}^{k-1}$, or $y\in N_{j}^{k-1}$.
We have, as observed at the beginning of the proof,
\begin{equation*}
\begin{split}
\|D\Phi_{k}\|_{Q_{j}^{k-1}\cup N_{j}^{k-1}}
&\leq \max\{\|D\Phi_{k}\|_{N_{j}^{k-1}},\|D\Phi_{k}\|_{\breve{Q}_{j}^{k-1}},\|D\Phi_{k}\|_{\Q_{k}}\}\\
&=\max\{1,C(n)\lambda_{k},1\}=C(n)\lambda_{k},
\end{split}
\end{equation*}
Since $x,y\in Q_{j}^{k-1}\cup N_{j}^{k-1}$, the mean value theorem yields
$$
|\Phi_{k}(x)-\Phi_{k}(y)|\leq C(n)\lambda_{k}|x-y|
$$
By Lemma \ref{ap:est l}, $\lambda_{k}\leq C(\phi)\lambda_{m}$, which gives
\begin{equation}
\label{est2 by am-1}
|\Phi_{k}(x)-\Phi_{k}(y)|\leq C_{1}(n,\phi)\lambda_{m}\frac{\beta_{m}}{10}=C(n,\phi)\alpha_{m-1}.
\end{equation}
Now, Lemmata \ref{ap:est ak}, \ref{ap:est bk} and inequality \eqref{conc psi} yield
\begin{eqnarray*}
\alpha_{m-1}
&\leq&
2^{-(m-1)}=2^{N+2}2^{-(N+(m+1))}\leq 2^{N+2}\phi(4\beta_{m+1})\\
&\leq&
40\cdot 2^{N+2}\phi(\beta_{m+1}/10)\leq 40\cdot 2^{N+2}\phi(|x-y|).
\end{eqnarray*}
This, combined with the estimates in \eqref{est by am-1} and \eqref{est2 by am-1}, proves the estimate \eqref{moc est}.

Note that the inverse mapping $\Phi_{\ell}^{-1}$ is constructed in exactly the same way as $\Phi_{\ell}$; the only difference is that in
Lemma~\ref{est:lip}, in place of the diffeomorphism $F$ exchanging `top' and `bottom layer' cubes of edge length $\frac{1}{4}$, we use
its inverse $F^{-1}$, which obviously possesses the same properties as $F$ (it is a diffeomorphism that is identity near $\partial Q$ and it
exchanges `top' and `bottom layer' cubes of edge length $\frac{1}{4}$, together with their $\frac{1}{40}$-tubular neighborhoods). Therefore, an
estimate analogous to \eqref{moc est} holds for $\Phi_{\ell}^{-1}$ as well (possibly with a different constant):
$$
|\Phi_{\ell}^{-1}(x)-\Phi_{\ell}^{-1}(y)|\leq C(n,\phi) \phi(|x-y|).
$$
Passing to the limit in the uniform metric $d$ in the space of homeomorphisms we see that if $\Phi=\lim_{\ell\to\infty} \Phi_{\ell}$, then
$$
|\Phi(x)-\Phi(y)|\leq C(n,\phi) \phi(|x-y|)
$$
and
$$
|\Phi^{-1}(x)-\Phi^{-1}(y)|\leq C(n,\phi) \phi(|x-y|).
$$
\hfill $\Box$

\begin{remark}
By the definition of $\beta_{k}$ (cf. Lemma \ref{ap:est bk}) one immediately sees that the convergence condition (2) on $\phi$ in
Proposition~\ref{prop:main} is natural and necessary for our construction: (2) holds if and only if the series $\sum 2^{k} \beta_{k}$ is convergent.
Recall (Figure~\ref{fig:Phi2}) that $\beta_{k}$ is the distance from the boundary of the $(k-1)$-th generation cube to the $k$-th generation cube,
thus $4 \beta_{1}=1-2\alpha_{1}<1$, $4\beta_{1}+8\beta_{2}=1-4\alpha_{2}<1$ and so on,
$$
2\beta_{1}+4\beta_{2}+8\beta_{3}+\cdots+2^{k}\beta_{k}=\frac{1}{2}\left(1-2^{k}\alpha_{k}\right)<\frac{1}{2},
$$
thus (2) is necessary for the sequence $\alpha_{k}$ to be well adapted for our construction.
\end{remark}

\bigskip

\section{Proof of Theorem \ref{thm:main}}
\label{Pr1.2}

Let $\phi$ satisfy conditions (1), (2) and (3) of Theorem~\ref{thm:main} and let $\psi$ be obtained from $\phi$ according to the construction given in the proof of Lemma \ref{better}.
In particular, constants dependent on the choice of $\psi$ depend on $\phi$ only.

The starting point of our iterative construction is the homeomorphism $F_{1}=\Phi$, given by Proposition~\ref{prop:main}, but with estimates dependent on $\psi$ instead of $\phi$.
Then $F_{1}$ is a.e. approximately differentiable in $Q$, it has the Lusin property, $F_{1}=\id$ near $\partial Q$ and there exists a compact set $C_{1}=A$ of positive measure
such that the Jacobian $J_{F_{1}}$ is negative (equal $-1$) in almost every point of $C_{1}$; $F_{1}$ is a diffeomorphism outside $C_{1}$.

Moreover, for any $x,y\in Q$,
$$
|F_{1}(x)-F_{1}(y)|+|F_{1}^{-1}(x)-F_{1}^{-1}(y)|\leq C(n,\psi)\psi(|x-y|)\leq C(n,\phi)\phi(|x-y|),
$$
with the last inequality true for some $C(n,\phi)$ thanks to (3), Lemma \ref{better}. We can assume that $C(n,\phi)\geq 6$.

Assume thus, for the inductive step of the construction, that we already have a homeomorphism $F_{k}:Q\to Q$ satisfying
\begin{enumerate}
\item[a)] $F_{k}$ is equal to identity near $\partial Q$,
\item[b)] $F_k$ has the Lusin property,
\item[c)] there exists a compact set $C_{k}\subset Q$ such that
\begin{itemize}
\item $|C_{k}|>0$,
\item for a.e. $x\in C_{k}$, the homeomorphism $F_{k}$ is approximately differentiable at $x$ and its Jacobian $J_{F_{k}}(x)$ is negative,
\item $F_{k}$ is a diffeomorphism outside $C_{k}$.
\end{itemize}
\item[d)]  for any $x,y\in Q$,
\begin{align*}
|F_{k}(x)-F_{k}(y)|&\leq \Big(1+\frac{1}{2}+\cdots+\frac{1}{2^{k}}\Big)C(n,\phi)\phi(|x-y|),\\
|F_{k}^{-1}(x)-F_{k}^{-1}(y)|&\leq \Big(1+\frac{1}{2}+\cdots+\frac{1}{2^{k}}\Big)C(n,\phi)\phi(|x-y|).
\end{align*}
\end{enumerate}

In the inductive step we construct $F_{k+1}$ by modifying $F_k$ in sufficiently
small balls outside $C_{k}$, enlarging the set of points at which the approximate Jacobian is negative.
To this end, let us choose an open set $\Omega$ such that $\Omega\Subset Q\setminus C_{k}$,
$|\Omega|>\frac{3}{4} |Q\setminus C_{k}|$ and $|F_{k}(\Omega)| >\frac{3}{4} |Q\setminus F_{k}(C_{k})|$.

Set
$$
M=1+\|DF_{k}\|_{\Omega}+\|(DF_{k})^{-1}\|_{\Omega}+\|D^{2}F_{k}\|_{\Omega}<\infty.
$$
We fill $\Omega$ with a finite family of pairwise disjoint balls $\{B_{i}\}_{i\in I}=\{B(x_i,r_i)\}_{i\in I}$, $B_{i}\Subset\Omega$,
with sufficiently small radii $r_i$, i.e. $r_{i}<\rho$, with $\rho$ to be determined later, in such a way that
$$
\big|\bigcup_{i}B_{i}\big|>\frac{2}{3}|\Omega|\quad\text{ and }\quad \big|\bigcup_{i}F_{k}(B_{i})\big|>\frac{2}{3}|F_{k}(\Omega)|.
$$

Assume $\rho<\left(10 (M+1)^{2}2^{k+1}\right)^{-1}$.
Then we can modify  $F_{k}$ according to Lemma \ref{goodapprox} inside each of the balls $B_{i}$, obtaining an approximately differentiable homeomorphism $\tilde{F}_{k}$ which coincides with $F_{k}$ on some neighborhoods of $\partial Q$ and of $C_{k}$ and which is affine on each of the balls $D_{i}$, concentric with $B_{i}$, but with radius $r_{i}/2$.
Namely $\tilde{F}_k(x)=T_i(x)=F_k(x_i)+DF_k(x_i)(x-x_i)$ for $x\in D_i$. Obviously, $\tilde{F}_k$ is an orientation preserving diffeomorphism in $Q\setminus C_k$.
In particular, the affine maps $T_i$ are orientation preserving, with $J_{T_i}=J_{F_k}(x_i)>0$.

Next, inscribe an $n$-dimensional cube $Q_{i}$, with edges parallel to the coordinate directions,
into each of the balls $D_{i}$ and denote by $S_{i}:Q_{i}\to Q$ the standard similarity (scaling + translation) transformation between $Q_{i}$ and the unit cube $Q$.  We define
$$
F_{k+1}=
\begin{cases}
\tilde{F}_{k} &\text{ in } Q\setminus \bigcup_{i\in I} Q_{i}\\
\tilde{F}_{k} \circ S_{i}^{-1}\circ \Phi \circ S_{i} & \text{ in each of the }Q_{i}.
\end{cases}
$$
Then in each of the cubes $Q_{i}$ there is a positive measure Cantor set $A_{i}$ such that
\begin{itemize}
\item $S_{i}^{-1}\circ\Phi\circ S_{i}|_{A_{i}}$ is a symmetry,
\item outside $A_{i}$,  $S_{i}^{-1}\circ\Phi\circ S_{i}$ is a diffeomorphism of $Q_{i}$ onto itself, equal to the identity near $\partial Q_{i}$.
\end{itemize}

The homeomorphism $F_{k+1}$ on $A_{i}$ is a composition of a symmetry
(orientation-reversing affine map with the Jacobian equal $-1$) and the orientation preserving affine tangent mapping $T_{i}$ with the Jacobian equal
$J_{F_{k}}(x_{i})>0$.
Thus $J_{F_{k+1}}$ is negative in all the
density points of $A_{i}$ (in fact it is constant on the set of density points of $A_{i}$: for a.e. $x\in A_{i}$ it is equal to $-J_{F_{k}}(x_{i})$).

Note that for each $i$, $|A_i|/|Q_i|=|A|/|Q|=|A|$ so
$|A_{i}|=|A||Q_{i}|$, where $A$ is the  Cantor set constructed in Proposition~\ref{prop:main}.  Note that $|A|$
depends on the dimension $n$ and the choice of $\psi$ and hence
it depends on $n$ and $\phi$ only. Thus
\begin{equation}
\label{eq:exhaust}
\begin{split}
\big| \bigcup_{i} A_{i} \big|= & \sum_{i}|A_{i}|=|A|\sum_{i}|Q_{i}|=C(n,\phi)\sum_{i}|B_{i}|\\
& \geq C(n,\phi)\frac{2}{3}|\Omega|\geq C(n,\phi)\frac{2}{3}\cdot\frac{3}{4}|Q\setminus C_{k}|=\frac{1}{2}\, C(n,\phi)|Q\setminus C_{k}|.
\end{split}
\end{equation}
We define $C_{k+1}=C_{k}\cup \bigcup_{i} A_{i}$. It easily follows from \eqref{eq:exhaust} that
\begin{equation}
\label{B52}
\Big|Q\setminus\bigcup_k C_k\Big|=0.
\end{equation}
Clearly, the homeomorphism $F_{k+1}$ has the corresponding properties a), b) and c).
The next step is proving the continuity estimate d) for $F_{k+1}$ and $F_{k+1}^{-1}$.
The arguments are repetitive, thus we provide the details in the most complex cases and sketch the remaining ones.

We first prove the estimate for the intermediate step $\tilde{F}_{k}$.
For points $x,y\in Q$ we have to consider several cases.

Let $x,y\not \in \bigcup_{i} B_{i}$. Then we have, by assumption,
$$
|\tilde{F}_{k}(x)-\tilde{F}_{k}(y)|=|F_{k}(x)-F_{k}(y)|\leq \Big(1+\frac{1}{2}+\cdots+\frac{1}{2^{k}}\Big)C(n,\phi)\phi(|x-y|).
$$
Similarly, for any $x,y\not \in \bigcup_{i}F_{k}(B_{i})$,
$$
|\tilde{F}^{-1}_{k}(x)-\tilde{F}^{-1}_{k}(y)|=|F_{k}^{-1}(x)-F_{k}^{-1}(y)|\leq \Big(1+\frac{1}{2}+\cdots+\frac{1}{2^{k}}\Big)C(n,\phi)\phi(|x-y|).
$$

The remaining cases when at least one of the points is in a ball $B_i$ are more difficult.
By (d), Lemma~\ref{goodapprox}, the mapping $\tilde{F_{k}}$ is bi-Lipschitz in each of the closed balls $\overline{B}_{i}$, with
bi-Lipschitz constant $\Lambda=2M$.

Assume that $\rho$ is such that for $t<2\Lambda\rho$ we have  $t/\phi(t)<\Lambda^{-1}2^{-k-1}$ (recall that by (b),
Lemma~\ref{ap:estpsi}, $t/\phi(t)\to 0$ with $t\to 0^{+}$).
Note that $\rho$ depends on $n$, $\phi$ and $k$ only.

Let $x,y\in \overline{B}_{i}$. Then $|x-y|<2\rho<2\Lambda\rho$, thus
\begin{equation}
\label{inaball}
|\tilde{F}_{k}(x)-\tilde{F}_{k}(y)|\leq \Lambda|x-y|\leq \frac{\phi(|x-y|)}{ 2^{k+1}}.
\end{equation}
In the same way we prove that if $x,y\in F_{k}(\overline{B}_{i})$, then $|x-y|<2\Lambda\rho$ and again
$$
|\tilde{F}^{-1}_{k}(x)-\tilde{F}^{-1}_{k}(y)|\leq \Lambda|x-y|\leq \frac{\phi(|x-y|)}{ 2^{k+1}}.
$$

Let $x\in B_{i}$, $y\in B_{j}$, $i\neq j$. We note that the segment $[x,y]$ must intersect the boundaries of
$B_{i}$ and $B_{j}$; let $z\in [x,y]\cap\partial B_{i}$ and $w\in [x,y]\cap\partial B_{j}$.
We have then $\tilde{F}_{k}(z)=F_{k}(z)$ and $\tilde{F}_{k}(w)=F_{k}(w)$; thus, by \eqref{inaball}, the inductive assumption,
and the fact that $\phi$ is increasing,
\begin{equation}
\label{intwoballs}
\begin{split}
|\tilde{F}_{k}(x)-\tilde{F}_{k}(y)| &\leq |\tilde{F}_{k}(x)-\tilde{F}_{k}(z)| +|\tilde{F}_{k}(z)-\tilde{F}_{k}(w)| +|\tilde{F}_{k}(w)-\tilde{F}_{k}(y)|\\
&\leq
\frac{\phi(|x-z|)}{ 2^{k+1}}+|F_{k}(z)-F_{k}(w)|+ \frac{\phi(|w-y|)}{ 2^{k+1}}\\
&\leq
\Big(1+\frac{1}{2}+\cdots+\frac{1}{2^{k}}\Big)C(n,\phi)\phi(|x-y|)+ \frac{\phi(|x-y|)}{ 2^{k}}.
\end{split}
\end{equation}
The estimates for $\tilde{F}_{k}^{-1}$ when $x\in {F}_k(B_i)$, $y\in {F}_k(B_j)$, $i\neq j$, are done in exactly the same manner.

Let $x\in B_{i}$ and $y\not\in \bigcup_i B_{i}$. This case is settled in the same way:
we decompose the segment $[x,y]$ into $[x,z]\cup[z,y]$, where $z\in\partial B_{i}$ and use the triangle inequality;
also the estimates for $\tilde{F}_{k}^{-1}$ are done in exactly the same manner.

Ultimately, we obtain that for any $x,y\in Q$
\begin{equation}
\label{estfortilde}
|\tilde{F}_{k}(x)-\tilde{F}_{k}(y)| \leq \Big(1+\frac{1}{2}+\cdots+\frac{1}{2^{k}}\Big)C(n,\phi)\phi(|x-y|)+ \frac{\phi(|x-y|)}{ 2^{k}},
\end{equation}
and the same estimate for $\tilde{F}_{k}^{-1}$.

Now, let us turn to $F_{k+1}$ and $F_{k+1}^{-1}$.

Outside the union of the cubes $Q_{i}$ we have that $F_{k+1}$
coincides with $\tilde{F}_{k}$ so d) follows from \eqref{estfortilde} and the fact that $C(n,\phi)>2$.

For $x,y\in Q_{i}$ we have, by Proposition~\ref{prop:main},
\begin{equation}
\label{r001}
\begin{split}
|F_{k+1}(x)-F_{k+1}(y)|&=|\tilde{F}_{k}\circ S_{i}^{-1}\circ \Phi\circ S_{i}(x)-\tilde{F}_{k}\circ S_{i}^{-1}\circ \Phi\circ S_{i}(y)|\\
&\leq C(n,\psi)\Lambda \lambda^{-1}\psi(\lambda |x-y|)\leq C(n,\psi)\Lambda\psi(|x-y|),
\end{split}
\end{equation}
where $\lambda>1$ is the scaling factor between $Q_{i}$ and the unit cube.
We used here the fact that $\tilde{F}_k$ is $\Lambda$-Lipschitz on $Q_i$.
The last inequality follows from (a), Lemma~\ref{ap:estpsi}.
In the same way we prove that whenever $x,y\in F_{k}(Q_{i})$, we have
\begin{equation}
\label{r002}
|F^{-1}_{k+1}(x)-F^{-1}_{k+1}(y)|\leq \frac{C(n,\psi)}{\lambda}\psi(\lambda \Lambda|x-y|)\leq C(n,\psi)\Lambda\psi(|x-y|).
\end{equation}
Assume, in addition to the previous restrictions on $\rho$, that
$$
\frac{\psi(t)}{\phi(t)}<\frac{1}{C(n,\psi)\Lambda 2^{k}}
\qquad
\text{for $t<2\Lambda\rho$,}
$$
where $C(n,\psi)$ is the same constant as the one in inequalities \eqref{r001} and \eqref{r002}. It is easy to see that we can find
such $\rho$ depending on $n$, $\phi$ and $k$ only. Recall that if $x,y\in Q_i\subset\overline{B}_i$ or
$x,y\in F_k(Q_i)\subset F_k(\overline{B}_i)$, then $|x-y|<2\Lambda\rho$ so
\begin{equation}
\label{zpsidofi}
|F_{k+1}(x)-F_{k+1}(y)|\leq C(n,\psi)\Lambda\psi(|x-y|)\leq \frac{\phi(|x-y|)}{2^{k}}
\end{equation}
and similarly
$$
|F_{k+1}^{-1}(x)-F_{k+1}^{-1}(y)|\leq \frac{\phi(|x-y|)}{2^{k}}\, .
$$
These estimates imply d).

Assume now that $x\in Q_{i}$, $y\in Q_{j}$. Then the segment $[x,y]$ intersects $\partial Q_{i}$ and $\partial Q_{j}$;
let $z\in [x,y]\cap\partial Q_{i}$ and $w\in [x,y]\cap\partial Q_{j}$. Note that $F_{k+1}(z)=\tilde{F}_{k}(z)$ and $F_{k+1}(w)=\tilde{F}_{k}(w)$.
Proceeding exactly as in \eqref{intwoballs}, we use the triangle inequality,  \eqref{estfortilde} and \eqref{zpsidofi} to get
\begin{eqnarray*}
\lefteqn{|F_{k+1}(x)-F_{k+1}(y)|}\\
&\leq& \frac{\phi(|x-z|)}{2^k}
 +
\Big[\Big(1+\frac{1}{2}+\ldots+\frac{1}{2^k}\Big)C(n,\phi)\phi(|z-w|)+ \frac{\phi(|z-w|)}{2^k}\Big] +\frac{\phi(|w-y|)}{2^k}\\
& \leq &
\Big(1+\frac{1}{2}+\ldots+\frac{1}{2^{k+1}}\Big)C(n,\phi)\phi(|x-y|),
\end{eqnarray*}
because $3/2^k<C(n,\phi)/2^{k+1}$.

The same arguments prove the above estimate in the case when $x\in Q_{i}$, $y\not\in \bigcup_j Q_{j}$.
The estimate for the inverse function $F_{k+1}^{-1}$
follows in exactly the same manner. This completes the proof of the inequalities in d) for $k+1$.

We proved that for all $x,y\in Q$ and all $k$
\begin{equation}
\label{last}
|F_k(x)-F_k(y)|+|F_k^{-1}(x)-F_k^{-1}(y)|\leq 4C(n,\phi)\phi(|x-y|).
\end{equation}
One can show as in \cite{GH} that the sequence $\{F_k\}$ converges in the uniform metric \eqref{um} to a homeomorphism $F$ that has
all properties listed in Theorem~\ref{thm:main}, but the Lusin property (N). However, instead of referring to \cite{GH} we will use a straightforward argument
showing convergence of a {\em subsequence} of $\{ F_k\}$. We proved in \eqref{last} that both families $\{F_k\}$ and $\{ F_k^{-1}\}$ are equicontinuous.
Since the families are bounded, it follows from the Arzel\`a-Ascoli theorem that subsequences converge uniformly
$F_{k_i}\rightrightarrows F$ and $F_{k_i}^{-1}\rightrightarrows G$. Since $\id = F_{k_i}\circ F_{k_i}^{-1}\rightrightarrows F\circ G$
we conclude that $F\circ G=\id$ so $F$ is a homeomorphism and that $F_{k_i}^{-1}\rightrightarrows F^{-1}$. Clearly $F|_{\partial Q}=\id$.
Passing to the limit in \eqref{last} gives
$$
|F(x)-F(y)|+|F^{-1}(x)-F^{-1}(y)|\leq 4C(n,\phi)\phi(|x-y|).
$$
It follows from the construction that for $m\geq k$, $F_m|_{C_k}=F_k|_{C_k}$ so $F|_{C_k}=F_k|_{C_k}$.
Since on the set $C_k$ the mapping $F=F_k$ has the Lusin property, it is approximately differentiable, and
$J_{F_k}=J_F<0$ a.e. in $C_k$, it follows from \eqref{B52} that $F$ is approximately differentiable a.e. with
$J_F<0$ a.e.

Moreover, $F$ has the Lusin property on the set $\bigcup_k C_k$, and it remains to show that
$|F(Q\setminus \bigcup_k C_k)|=0$. Equivalently, we need to show that
$$
|Q\setminus F(C_k)|\to 0
\quad
\text{as $k\to\infty$.}
$$
Since $C_{k+1}=C_k\cup\bigcup_i A_i$, it suffices to show that there is a constant $C>0$, depending on $n$ and $\phi$ only, such that
$$
\Big| F\Big(\bigcup_i A_i\Big)\Big| \geq C|Q\setminus F(C_k)|=C|Q\setminus F_k(C_k)|.
$$
Recall that
$$
\Big|\bigcup_i F_k(B_i)\Big|> \frac{2}{3} |F_k(\Omega)| > \frac{1}{2} |Q\setminus F_k(C_k)|,
$$
so it suffices to show that $|F(A_i)|\geq C|F_k(B_i)|$.

Let $E_i=\overline{B}(x_i,2r_i)$. According to (b), Lemma~\ref{goodapprox},
$F_k(B_i)=\tilde{F}_k(B_i)\subset T_i(E_i)$. Since $|T_i(E_i)|=2^n|T_i(B_i)|$, we get $|F_k(B_i)|\leq 2^n |T_i(B_i)|$.

Observe also that $F(A_i)=T_i(A_i)$, so
$$
\frac{|F(A_i)|}{|T_i(B_i)|} =
\frac{|T_i(A_i)|}{|T_i(B_i)|} =
\frac{|A_i|}{|B_i|} =
C(n)\, \frac{|A_i|}{|Q_i|}=C(n)|A|= C(n,\phi)
$$
and hence
$$
|F(A_i)|=C(n,\phi)|T_i(B_i)|\geq
2^{-n} C(n,\phi) |F_k(B_i)|.
$$
This completes the proof of the Lusin property of the homeomorphism $F$.
\hfill $\Box$

\end{document}